\documentclass[a4paper,twoside,11pt]{article}

\usepackage{amsfonts}
\usepackage{mathrsfs}
\usepackage{amsmath}
\usepackage{amsmath,graphics}
\usepackage{amssymb}
\usepackage{indentfirst,latexsym,bm,amsmath,pstricks,amssymb,amsthm,graphicx,fancyhdr,float}
\usepackage[all]{xy}
\usepackage{amsthm}
\usepackage{amscd}
\usepackage[colorlinks,urlcolor=black]{hyperref}
\usepackage[center]{titlesec}
\titleformat{\section}{\centering\Large\bfseries}{\arabic{section}.}{1em}{}
\titleformat{\subsection}{\centering\large\bfseries}{\arabic{section}.\arabic{subsection}.}{1em}{}

\begin{document}
\theoremstyle{plain}
\newtheorem{thm}{Theorem}
\newtheorem{cor}[thm]{Corollary}
\newtheorem{theorem}{Theorem}[section]
\newtheorem{prop}[theorem]{Proposition}
\newtheorem{lemma}[theorem]{Lemma}
\newtheorem{corollary}[theorem]{Corollary}
\newtheorem{proposition}[theorem]{Proposition}
\newtheorem{conjecture}[theorem]{Conjecture}
\newtheorem{problem}[theorem]{Problem}
\theoremstyle{definition}
\newtheorem{remark}[theorem]{Remark}
\newtheorem{remarks}[theorem]{Remarks}
\newtheorem{definition}[theorem]{Definition}
\newtheorem{claim}{Claim}[theorem]
\newtheorem{example}[theorem]{Example}
\newtheorem{news}[theorem]{}

\numberwithin{equation}{section}
\def\wt{\widetilde}
\def\ol{\overline}
\def\ra{\rightarrow}
\def\lra{\longrightarrow}
\def\div{\text{\rm{div\,}}}
\def\Div{\text{\rm{Div\,}}}
\def\Aut{\text{\rm{Aut\,}}}
\def\({$($}
\def\){$)$}
\def\chit{\chi_{\rm top}}
\def\bbp{\mathbb P}
\def\calh{\mathcal H}
\def\cala{\mathcal A}
\def\calt{\mathcal T}
\def\calm{\mathcal M}
\def\caln{\mathcal N}
\def\cals{\mathcal S}
\def\calf{\mathcal F}
\def\Pic{\text{{\rm Pic\,}}}
\def\rank{\text{{\rm rank\,}}}

\font\eightrm=cmr10 at 8pt \font\bigrm=cmr10 scaled1440
\font\midrm=cmr10 scaled1200 \font\ninerm=cmr10 at 9pt
\font\ninebf=cmbx10 at 9pt \font\bbf=msbm10 \font\bigbf=cmbx10
scaled1440 \font\Bigbf=cmbx10 scaled2073 \font\midbf=cmbx10
scaled1200 \font\eightit=cmti8 \font\nineit=cmti9
\renewcommand{\labelenumi}{$($\arabic{enumi}$)$}
\renewcommand{\thethm}{\Alph{thm}}
\renewcommand{\theequation}{\arabic{section}-\arabic{equation}}
\renewcommand{\labelenumii}{(\Alph{enumi})}
\pagestyle{fancy}
\setlength{\headsep}{0.4cm}
\fancyhf{} 
\chead{{\bf On Shimura curves in the Torelli locus of curves}
        \hfill {{\bf X. Lu} \& {\bf K. Zuo}}}
\cfoot{---\thepage---}

\title{On Shimura curves in the Torelli locus of curves}
\author{Xin Lu\qquad Kang Zuo}
\date{}
\renewcommand{\thefootnote}{}
\footnotetext{This work is supported by SFB/Transregio 45 Periods, Moduli Spaces and Arithmetic of Algebraic Varieties of the DFG (Deutsche Forschungsgemeinschaft),
and partially supported by National Key Basic Research Program of China (Grant No. 2013CB834202).}

\maketitle

\phantomsection
\begin{abstract}
\addcontentsline{toc}{section}{Abstract}
Oort has conjectured that there do not exist Shimura curves
lying generically in the Torelli locus of curves of genus $g \geq 8$.
We show that there do not exist one-dimensional Shimura families of semi-stable curves of genus $g\geq 5$
of Mumford type. We also show that there do not exist
Shimura curves lying generically in the Torelli locus of hyperelliptic curves of genus $g\geq 8$.
The first result proves a slightly weaker form of the
conjecture for the case of Shimura curves of Mumford type.
The second result proves the conjecture for the Torelli locus of hyperelliptic curves.
We also present examples of Shimura curves contained
generically in the Torelli locus of curves of genus $3$ and $4$.
\end{abstract}

\section{Introduction}\label{sectionintroduction}
Let $\calm_g$ be the moduli space of smooth projective curves of genus $g\geq 2$ over $\mathbb C$,
and $\cala_g$ be the moduli space of $g$-dimensional principally polarized  abelian varieties over $\mathbb C$.
There is a natural morphism called the {\it Torelli morphism},
$$j:\,\calm_g\to \cala_g,$$
which sends a curve to its
canonically principally polarized Jacobian.
The image $\calt_g^{\rm o}$ is called the {\it open Torelli locus},
and its Zariski closure $\calt_g \subseteq \cala_g$ is called the {\it Torelli locus}.

According to a conjecture of Coleman, for a fixed genus $g\geq 4$, there are only finitely many
CM-points in $\calm_g$.
This conjecture is known to be false for $4\leq g\leq 7$,
by the fact that there exist Shimura subvarieties $Z$
of positive dimension {\it contained generically in the Torelli locus},
i.e.,
$$Z \subseteq \calt_g \text{\quad and \quad} Z\cap \calt_g^{\rm o}\neq \emptyset.$$
We refer to \cite{mo11} for a beautiful discussion of this topic.
Combining with the conjecture of Andr\'e-Oort, which says that a Shimura variety is
characterized by having a dense subset of CM-points,
one has the following expectation (cf. \cite[\S\,5]{oort97}, see also \cite[\S\,4]{mo11}):

\begin{conjecture}[Oort]\label{conjectureShimura}
For large $g$ (in any case $g\geq 8$), there does not exist a Shimura subvariety
of positive dimension contained generically in the Torelli locus.
\end{conjecture}

We study here the problem for the case of Shimura curves,
and in general Kuga curves.
Recall that a closed subvariety
$$Z\hookrightarrow \cala_g = Sp_{2g}(\mathbb Z)\backslash Sp_{2g}(\mathbb R)/ U(g)$$
is called a {\it Kuga subvariety} if the inclusion is induced by a homomorphism
$G \to Sp_{2g}$ for some algebraic group $G$ (cf. \cite{mumford69}).
Moreover $Z$ is a {\it Shimura subvariety} if $Z$ is Kuga and contains a CM-point.
A {\it Kuga} (resp. {\it Shimura{\rm)} curve} $U$ is a one-dimensional
connected Kuga (resp. Shimura) variety.
The corresponding universe family of abelian varieties over a Kuga (resp. Shimura) curve
is called a {\it Kuga} (resp. {\it Shimura{\rm)} family of abelian varieties}.

Let $h:\, X \to B$ be a semi-stable family of $g$-dimensional abelian varieties over a smooth
projective curve $B$, with singular fibres $\Upsilon_{nc}\,/\,\Delta_{nc}$.
Let $U:=B-\Delta_{nc}$ and $V:=h^{-1}(U)$. Then $h:\,V\to U$ is an abelian scheme and
the direct image sheaf $R^1h_*\mathbb C_{V}$ is a local system on $U$
which underlies a variation of polarized Hodge
structure $\mathbb V$ of weight one.
Let $$\left(E^{1,0}\oplus E^{0,1},~\theta\right):
=\left( h_*\Omega^1_{X/B}(\log\Upsilon_{nc})\oplus R^1h_*\mathcal O_X,\,\theta\right)$$
denote  the logarithmic Higgs bundle by taking the grading of the
Hodge filtration on the Deligne's extension of
the de Rham bundle $(R^1h_*\mathbb C_{V}\otimes\mathcal O_{U}, \nabla),$  where
the Higgs field $\theta$ is given by the edge morphism of the tautological sequence
\begin{equation*}\label{defoftheta}
0\lra f^*\Omega^1_B(\log \Delta_{nc}) \lra \Omega^1_X(\log \Upsilon_{nc})
\lra \Omega^1_{X/B}(\log \Upsilon_{nc})\lra 0.
\end{equation*}
By \cite{fujita78} or \cite{kollar87}, $E$ is decomposed as a direct sum
$$\left(E,\theta\right)
=\left(A^{1,0}\oplus A^{0,1},~\theta|_{A^{1,0}}\right)\oplus \left(F^{1,0}\oplus F^{0,1},~ 0\right) $$
of Higgs bundles, where   $A^{1,0}$ is an ample vector bundle,
$F^{1,0}$ and $F^{0,1}$ are vector bundles  underlying  unitary
local subsystems  $\mathbb F^{1,0}\oplus \mathbb F^{0,1}\subset \mathbb V.$

Following \cite{vz03}, the Higgs field is called  to be {\it maximal} if
$$\theta|_{A^{1,0}}:\,A^{1,0} \to A^{0,1}\otimes \Omega^1_B(\log \Delta_{nc})$$
is an isomorphism, and {\it strictly maximal} if furthermore $F=0$.

The Arakelov inequality (cf. \cite{faltings83} or \cite{jz02}) says that
$$\deg A^{1,0}\leq \frac{\rank A^{1,0}}2\cdot\deg\Omega^1_{B}(\log\Delta_{nc}).$$
We say that the family of abelian varieties $X\to B$ {\it reaches the Arakelov bound} if
the above inequality becomes an equality.
It is shown in \cite{vz03} that this property is equivalent to the maximality of the Higgs field for $A$,
i.e., $\theta|_{A^{1,0}}$ is an isomorphism.

It is proved in \cite{vz04} by Viehweg and the second author that if $h: V\to U$
has a strictly maximal Higgs field, then $h:\, V\to U$ is a universal family over
a {\it Shimura curve of Mumford type},
which means that if $\Delta_{nc}\neq\emptyset$, then
$V$ is isogenous over $U$ to a $g$-tuple self-product of
a universal family of elliptic curves;
if $\Delta_{nc}=\emptyset$, then $h$ is derived from the corestriction of a quaternion division algebra
over a totally real number field with all infinite places ramified except one
(see \cite{vz04} for more details). Moreover, M\"{o}ller showed in \cite{moller11}
that the converse also holds. In general,
it is showed in \cite{mvz12} that $h: V\to U$ is a Kuga family if and only if it has a maximal Higgs field.

In \cite{hain99}, Hain studied locally symmetric families of compact Jacobians
satisfying some additional conditions. Based on his
methods, de Jong and Zhang (\cite{djongzh07}) proved that certain
Shimura subvarieties parameterizing abelian varieties with real multiplication
do not lie generically in $\calt_g$ for $g\geq 4$.
In \cite{djongnoot89},
de Jong and Noot developed a method based on a criterion due to  Dwork-Ogus
using $p-$adic Hodge theory (\cite{do86}) and proved that the base varieties of  some specific
universal families  of curves arising from cyclic covers of $\mathbb P^1$ are not contained generically in
$\calt_g$.
Extending this to some general case, recently  Moonen (\cite{moonen10}) proved
that there are exactly twenty  families of curves coming from cyclic
covers of $\mathbb P^1$ such that the base varieties lie generically in $\calt_g$ with $g\leq 7$,
which implies that Conjecture\,\ref{conjectureShimura} holds if
the corresponding families arising from a universe cyclic cover of $\bbp^1$.
In \cite{kukulies10}, Kukulies showed that a given rational Shimura curve with strictly maximal
Higgs field in $\cala_g$ cannot be contained generically in $\calt_g$ for $g$ sufficiently large.
\vspace{0.15cm}

Our first result is to exclude certain Shimura curves arising from families of curves
with strictly maximal Higgs field.
We prove an effective bound on the genus $g$ for which there exists a Shimura
family of curves of genus $g$ with strictly maximal Higgs field.

Let $f: S\to B$ be a family of semi-stable curves over a smooth projective curve $B$
and let $\Delta_{nc}\subset B$ denote those points corresponding to fibres of $f$
with non-compact Jacobians. Put $U=B\setminus \Delta_{nc}$ and $S_0:=f^{-1}(U)$.
Then the relative Jacobian $$jac(f):~ Jac(S_0/U)\to U$$ is an abelian scheme over $U$. We call
$f$ to be a {\it Kuga} (resp. {\it Shimura{\rm)} family of curves},
if $jac(f)$ is a Kuga (resp. Shimura) family of abelian varieties.
The family $f$ is called to be with strictly maximal Higgs field,
if the Higgs field associated to $jac(f)$ is strictly maximal,
or equivalently if $jac(f)$
is a universal family over a Shimura curve $U$ of Mumford type by \cite{vz04}.
\begin{theorem}\label{mainthm1}
For $g\geq 5$,
there does not exist a Shimura family $f: S\to B$
of genus-$g$ curves with strictly maximal Higgs field.
\end{theorem}

Our next result is regarding Kuga and Shimura curves arising from families of hyperelliptic curves,
without the assumption on the strictly maximality of Higgs field.
Let
$$\mathcal H_g\subset \mathcal M_g$$ denote the moduli space of smooth hyperelliptic curves,
$$\mathcal H_g^{ct}\subset \overline{\mathcal M}_g$$
denote the moduli space of stable hyperelliptic curves with compact Jacobians and
$$j(\mathcal H_g)\subset j(\mathcal H^{ct}_g)\subset \calt_g$$
denote the images under the Torelli map. Note that the Zariski closure of $j(\mathcal H_g)$
in $\mathcal A_g$ is $j(\mathcal H^{ct}_g)$.

A Kuga or Shimura curve $U\subset \mathcal A_g$ is said to be {\it contained
generically in the Torelli locus $j(\calh^{ct}_g)$ of hyperelliptic curves},
if
$$U \subseteq j(\calh^{ct}_g)\text{\quad and \quad }U\cap j(\calh_g)\neq \emptyset.$$

It is clear that if $f$ is a Kuga (resp. Shimura) family of hyperelliptic curves,
then the image of $U$ under the Torelli map is a Kuga (resp. Shimura)
curve lying generically in  $j(\mathcal H_g^{ct}).$

Conversely, given a Kuga (resp. Shimura) curve $U$
lies generically in  $j(\mathcal H_g^{ct}).$
We like to show $U\subset j(\mathcal H_g^{ct})$ is induced by
a Kuga (resp. Shimura) family of hyperelliptic curves.

By taking an $n$-level structure, we may assume that $U\subseteq \calt_{g,n}\subseteq \cala_{g,n},$
hence, $U$ carries a universal family $h: V\to U$ of abelian varieties,
which is the Kuga (resp. Shimura) family of
abelian varieties over $U$ classifying the level $n$-structure.
We consider now the Torelli map
$$j: \mathcal M_{g,n}\to \mathcal A_{g,n}$$
By Oort-Steenbrink (cf. \cite{os79}), $j$ is a 2-to-1 morphism exactly outside the hyperelliptic locus
$\mathcal H_{g,n}\subset \mathcal M_{g,n}$. Furthermore, the restriction of $j$ to the hyperelliptic locus
$$ j: \mathcal H_{g,n}\to j(\mathcal H_{g,n})\subset  \mathcal A_{g,n}$$
is injective and immersion.  So, we may regard
$$ j: \mathcal H^{ct}_{g,n}\to j(\mathcal H^{ct}_{g,n})$$
as the blowing up along the subvariety $ j(\mathcal H^{ct}_{g,n})\setminus j(\mathcal H_{g,n}).$
Since  $U$ is a smooth  and closed  curve  in $j(\mathcal H^{ct}_{g,n})$ and
$U\cap j(\mathcal H_{g,n})\not=\emptyset$,
the proper transformation  $\hat U\subset \mathcal H^{ct}_{g,n}$
of $U$ under the blowing up $j$ is isomorphic to $U$
$$ j_{\hat U}: \hat U\simeq U.$$
Hence,  the pullback of the Kuga (resp. Shimura) family of abelian varieties
$$h: V\to U$$  under the isomorphism
$j_{\hat U}$
is again a Kuga (resp. Shimura) family
$$ j^*_{\hat U}(h): j^*_{\hat U}(V)\to \hat U$$
By the definition of the Torelli map $j$, it is
the Jacobian of the pullback to $\hat U\to \mathcal H^{ct}_{g,n}$ of
the universal family of stable hyperelliptic curves of compact type
to $\hat U\to \mathcal H^{ct}_{g,n},$  saying
$$f: S_0\to\hat U.$$

\begin{theorem}\label{mainthm}
For $g\geq 8$, there does not exist a Kuga family $f: S\to B$ of hyperelliptic curves.
In particular, for $g\geq 8$ there does not exist a Shimura  family $f: S\to B$ of hyperelliptic curves.
\end{theorem}
By the above discussion, as a consequence of Theorem \ref{mainthm}, we obtain
\addtocounter{theorem}{-1}
{\renewcommand{\thetheorem}{\arabic{section}.\arabic{theorem}$'$}
\begin{theorem}
For $g\geq 8$, there does not exist a  Kuga curve, which lies generically in the Torelli locus
$j(\mathcal H^{ct}_g)$ of hyperelliptic curves.
In particular, for $g\geq 8$ there does not exist a Shimura curve,
which lies generically in $j(\mathcal H^{ct}_g)$.
\end{theorem}   }

Let $f:\,S \to B$ be a semi-stable family of curves of genus $g\geq 2$.
Let $\Upsilon\to \Delta$
denote the semi-stable singular fibres,
$\Upsilon_{c}\to \Delta_{c}$ denote the singular fibres with compact Jacobians,
$\Upsilon_{nc}\triangleq\Upsilon\setminus\Upsilon_{c} \to \Delta_{nc}\triangleq\Delta\setminus\Delta_{c}$
correspond to singular fibres with non-compact
Jacobians, $U:=B\setminus\Delta_{nc}$, and $S_0:=f^{-1}(U)$.
Then the logarithmic
Higgs bundle associated to the VHS of the relative Jacobian
$$jac(f):~ Jac(S_0)\lra U$$
 is decomposed as Higgs bundles
$$(E^{1,0}\oplus E^{0,1},\theta)
=(A^{1,0}\oplus A^{0,1},~\theta|_{A^{1,0}})\oplus (F^{1,0}\oplus F^{0,1},~0),$$
 where
$$\theta|_{A^{1,0}}:~ A^{1,0}\lra A^{0,1}\otimes\Omega^1_B(\log\Delta_{nc})$$
is described on Page \pageref{defoftheta}.  Since the family $f: S\to B$ is semi-stable,
it is well known that
$$\left(E^{1,0}\oplus E^{0,1},\theta\right)\cong
\left(f_*\Omega^1_{S/B}(\log \Upsilon)\oplus R^1f_*\mathcal O_S,~\theta\right).$$

Theorem \ref{mainthm1} is then
a consequence of the following facts:\\[.1cm]
{\bf (i).}
Given a semi-stable family $f: S\to B$ of curves with strictly maximal Higgs field,
the Arakelov equality for the characterization of the relative Jacobian
$$jac(f):~ Jac(S_0)\lra U$$
to be a Shimura family becomes (cf. \cite{vz04}):
\begin{equation}\label{arakelovg}
\deg f_*\Omega^1_{S/B}(\log\Upsilon)= \deg A^{1,0}=
\frac{\rank A^{1,0}}2\cdot\deg\Omega^1_{B}(\log\Delta_{nc})
=\frac{g} 2\cdot\deg\Omega^1_B(\log\Delta_{nc}).
\end{equation}
{\bf (ii).}
The following two theorems on a non-isotrivial semi-stable family of curves.
\vspace{0.2cm}

Let $f:\,S \to B$ be as above.
For every singular fibre $F$, let
$\delta_i(F)$ be the number of nodes of $F$ of type $i$ ($1\leq i \leq [g/2]$),
where a node $q$ of $F$ is said to be of type $i$ ($1\leq i \leq [g/2]$), if
the partial normalization of $F$ at $q$ consists of two connected components
of arithmetic genera $i$ and $g-i$.
Let
\begin{equation}\label{defofssMM}
\left\{\begin{aligned}
~&\delta_i(\Upsilon)~\,=\sum_{F \in \Upsilon}\delta_i(F),\qquad
&&\delta_h(\Upsilon)~\,=\sum_{i=2}^{[g/2]} \delta_{i}(\Upsilon);\\
&\delta_i(\Upsilon_{c})=\sum_{F \in \Upsilon_{c}}\delta_i(F),\qquad
&&\delta_h(\Upsilon_{c})=\sum_{i=2}^{[g/2]}\delta_{i}(\Upsilon_{c}).
\end{aligned}\right.
\end{equation}

\begin{theorem}\label{thmlower1}
Let $f:\,S \to B$ be a non-isotrivial
semi-stable family of curves of genus $g\geq2$ as above,
and $\omega_{S/B}=\omega_{S}\otimes f^*\omega_B^{\vee}$ the relative canonical sheaf.
Then
\begin{equation}\label{eqnlower1}
\omega_{S/B}^2\geq \frac{4(g-1)}{g}\cdot\deg \left(f_*\Omega^1_{S/B}(\Upsilon)\right)
+\frac{3g-4}{g}\delta_1(\Upsilon)+\frac{7g-16}{g}\delta_h(\Upsilon).
\end{equation}
\end{theorem}

\begin{theorem}\label{thmupper}
Let $f:\,S \to B$ be the same as in Theorem {\rm\ref{thmlower1}}. Then
\begin{equation}\label{eqnupper}
\omega_{S/B}^2 \leq (2g-2)\cdot \deg\left(\Omega^1_{B}(\log\Delta_{nc})\right)
+2\delta_1(\Upsilon_{c})+3\delta_h(\Upsilon_{c}).
\end{equation}
Moreover, if $\Delta_{nc}\neq \emptyset$ or $\Delta=\emptyset$, then the above inequality is strict.
\end{theorem}
\noindent
Theorem \ref{thmlower1} is a direct consequence of Moriwaki's sharp slope inequality
(cf. \cite{moriwaki98}).
While Theorem \ref{thmupper} is based on Miyaoka's theorem (cf. \cite{miyaoka84})
for the bound on the number of quotient singularities in a surface.
The base change technique is also a key point.

The proof of Theorem \ref{mainthm} is much more complicated,
but the idea is the same as that of Theorem \ref{mainthm1}.
Without the assumption on the strictly maximality of Higgs field,
we have to bound the rank of the flat part of the Higgs bundle
and to give an analogous lower bound of $\omega_{S/B}^2$ for a semi-stable family $f:\,S \to B$
of hyperelliptic curves with positive relative irregularity $q_f:=q(S)-g(B)$.

\begin{theorem}\label{thmFtrivial}
Let $f:\,S \to B$ be as in Theorem {\rm\ref{thmlower1}}.
If $f$ is a hyperelliptic family, then
after passing to a finite \'etale cover of $B$, the local subsystems
$\mathbb F^{1,0}$ and $\mathbb F^{0,1}$ become to trivial local systems.
\end{theorem}
\noindent
So, by Theorem \ref{thmFtrivial} together with Deligne's global invariant cycle theorem
(cf. \cite[\S\,4.1]{deligne71}) or Fujita's decomposition theorem (cf. \cite[Theorem\,3.1]{fujita78}),
we show that the relative irregularity  $q_f$ is equal to $\rank F^{1,0}$
after passing a finite \'etale cover of $B$, i.e.,
$$q_f=\rank F^{1,0}=g-\rank A^{1,0}.$$
Thus the Arakelov equality (cf. \cite{mvz12}) for the characterization of the relative Jacobian
of $f:\,S\to B$ to be a Kuga family becomes
\begin{equation}\label{arakelovrankA}
\deg f_*\Omega^1_{S/B}(\log\Upsilon)
=\frac{g-q_f} 2\cdot\deg\Omega^1_B(\log\Delta_{nc}).
\end{equation}

\begin{theorem}\label{thmlower}
Let $f:\,S \to B$ be as in Theorem {\rm\ref{thmlower1}} and $q_f=q(S)-g(B)$ the relative irregularity.
Assume that $f$ is a hyperelliptic family. Then
\begin{eqnarray}
&&\hspace{-0.5cm}\omega_{S/B}^2~\geq~\frac{4(g-1)}{g-q_f}
\cdot\deg \left(f_*\Omega^1_{S/B}(\Upsilon)\right)+\label{eqnlower}\\[0.15cm]
&&\hspace{-0.5cm}
\left\{
\begin{aligned}
&\frac{3g^2-(8q_f+1)g+10q_f-4}{(g+1)(g-q_f)} \delta_1(\Upsilon)
+\frac{7g^2-(16q_f+9)g+34q_f-16}{(g+1)(g-q_f)} \delta_h(\Upsilon), &&
\text{if~} \Delta_{nc}\neq \emptyset;\\
&\sum_{i=1}^{[g/2]} \left(\frac{4(2g+1-3q_f)i(g-i)}{(2g+1)(g-q_f)}-1\right) \delta_i(\Upsilon), &&
\text{if~}\Delta_{nc}= \emptyset.
\end{aligned}\right.\nonumber
\end{eqnarray}
Moreover, if $\Delta_{nc}=\emptyset$ and $q_f\geq 2$, then
\begin{equation}\label{eqnlowerSigma_0=0}
\sum_{i=q_f}^{[g/2]} \frac{(2i+1)(2g+1-2i)}{g+1} \cdot \delta_i(\Upsilon)\geq
\sum_{i=1}^{q_f-1} 4i(2i+1) \cdot \delta_i(\Upsilon).
\end{equation}
\end{theorem}
\noindent
There are two ingredients in the proof of Theorem \ref{thmFtrivial}.
The first one is  Lemma \ref{lemmapeters} on the global invariant
cycle with unitary locally constant coefficient.
The second one is Bogomolov's lemma on Kodaira dimension of an invertible subsheaf
of the sheaf of logarithmic differential forms (cf. \cite{sakai80}).
The proof of Theorem \ref{thmlower} is based on formulas given by
Cornalba and Harris (cf. \cite{ch88}).
When $q_f>0$, the observation that the smooth double cover
induced by the hyperelliptic involution is fibred (cf. Proposition \ref{fibredprop})
plays a crucial role.
\vspace{.15cm}

In order to illustrate the idea how  Theorem \ref{mainthm1} (and Theorem \ref{mainthm})
follows from the above ingredients, we just consider here the simplest case:
$U$ is non-compact (i.e., $\Delta_{nc}\neq \emptyset$)
and the logarithmic Higgs bundle associated to the family has strictly maximal Higgs field.
By \eqref{arakelovg} together with \eqref{eqnlower1}, one has\label{ideaofproof}
$$\omega_{S/B}^2\geq (2g-2)\cdot \deg\left(\Omega^1_{B}(\log\Delta_{nc})\right)
+\frac{3g-4}{g}\delta_1(\Upsilon)+\frac{7g-16}{g}\delta_h(\Upsilon).$$
Note that \eqref{eqnupper} in Theorem \ref{thmupper} is strict if $\Delta_{nc}\neq \emptyset$.
Hence by Theorem \ref{thmupper}, we have
$$\omega_{S/B}^2 < (2g-2)\cdot \deg\left(\Omega^1_{B}(\log\Delta_{nc})\right)
+2\delta_1(\Upsilon_{c})+3\delta_h(\Upsilon_{c}).$$
By the definition,
\begin{equation*}
0\leq \delta_1(\Upsilon_{c}) \leq \delta_1(\Upsilon), \quad
\text{and}\quad 0\leq \delta_h(\Upsilon_{c}) \leq \delta_h(\Upsilon).
\end{equation*}
Combining all these together, one obtains
$$0> \frac{g-4}{g}\cdot\big(\delta_1(\Upsilon)+4\delta_h(\Upsilon)\big).$$
Because both $\delta_1(\Upsilon)$ and $\delta_h(\Upsilon)$ are non-negative,
it follows that $g< 4$, i.e.,
there does not exist a Kuga family $f: S\to B$
of curves of genus $g\geq 4$ with $\Delta_{nc}\neq \emptyset$ and strictly maximal Higgs field.
\vspace{0.2cm}

The paper is organized as follows:
In Section \ref{sectionprelim}, we introduce some notations and terminology.
\\[.1cm]
In Section \ref{sectionarbitrary}, we prove
Theorems \ref{thmlower1} and \ref{thmupper} for arbitrary semi-stable families.
Theorem \ref{thmlower1} is derived from Moriwaki's sharp slope inequality
(cf. \cite[Theorem D]{moriwaki98}) together with Noether's formula.
The proof of Theorem \ref{thmupper} is based on
a generalized Miyaoka-Yau's inequality (cf. \cite[Theorem\,1.1]{miyaoka84})
and the base change technique.
\\[.1cm]
In Section \ref{sectionlowerhyper}, we prove Theorem \ref{thmlower}.
When $q_f=0$, \eqref{eqnlower} is a direct consequence of
Cornalba-Harris' formula (cf. \cite[Proposition 4.7]{ch88}) together with Noether's formula.
When $q_f>0$,
the proof starts from the observation that the double cover
$\pi:\, S \to S/\langle\sigma\rangle$ induced by the hyperelliptic involution is fibred.
As a consequence, the branched divisor of $\pi$ is very special.
And the proof of Theorem \ref{thmlower} is completed in Section \ref{subsectpfthmlower}
by combining this with Cornalba-Harris' formula.
\\[.1cm]
Theorem \ref{thmFtrivial} is proved in Section \ref{sectionflathyper}, which is based on two ingredients.
The first one is  Lemma \ref{lemmapeters} on the global invariant
cycle with unitary locally constant coefficient,
which generalizes Deligne's original theorem with the constant coefficient.
The second one comes from Bogomolov's lemma on Kodaira dimension of an invertible subsheaf
in the sheaf of logarithmic differential forms on a smooth projective surface
(cf. \cite[Lemma 7.5]{sakai80}).
\\[.1cm]
In Section \ref{sectionconclusion}, we are  in the position to prove
Theorems \ref{mainthm1} and \ref{mainthm} with the idea demonstrated on Page \pageref{ideaofproof}.
Assume $f: S\to B$ is a Kuga family of curves of genus $g$.
Then the Arakelov inequality for $f$ becomes to an equality
$$ \deg f_*\Omega^1_{S/B}(\log\Upsilon)=\frac{\rank A^{1,0}}{2}\cdot\deg\Omega^1_B(\log\Delta_{nc}).$$
Theorems \ref{thmlower1} and \ref{thmupper} together with this equality
for $\rank A^{1,0}=g$ give rise to the required bound $g\leq 4$,
which proves Theorem \ref{mainthm1}.
If $f$ is hyperelliptic,
according to Theorems \ref{thmupper} and \ref{thmlower} together with the above equality
for $\rank A^{1,0}=g-q_f$ by Theorem \ref{thmFtrivial},
we obtain inequalities of
$g$ as a rational function of the variable $g$ in the different subcases.
We prove Theorem \ref{mainthm} by showing that
$g\leq 7$ in all subcases.
\\[.1cm]
Finally in Section \ref{sectionexample}, we present two examples of Shimura curves
contained generically in the Torelli locus of hyperelliptic curves of genus $g=3$
and $4$ respectively.
In particular, the Higgs field in the example of genus $g=4$ is strictly maximal,
which shows that the bound in Theorem\,\ref{mainthm1} is optimal.

\section{Preliminaries}\label{sectionprelim}
In this section, we introduce notations and terminology that will be used in the paper.

A curve $F$ is called {\it semi-stable} (resp. {\it stable}) if it is a reduced nodal curve,
and every smooth rational component intersects
the rest part of $F$ at least two (resp. three) points.
A morphism $f:\,S\to B$
is called a {\it semi-stable family} (resp. {\it stable family}) of curves of genus $g$,
if $f$ is a morphism from a projective surface $S$ to a smooth projective curve $B$
with connected fibres,
the general fibre is a connected nonsingular complex projective curve of genus $g$,
and all the singular fibres of $f$ are semi-stable (resp. stable).
In this paper, when we talk about a semi-stable family $f:\,S \to B$ as above,
we always assume the total surface $S$ is smooth.
If the general fibre of $f$ is a hyperelliptic curve,
then we call $f$ a {\it hyperelliptic family}.
$f$ is called {\it smooth} if all its fibres are smooth,
{\it isotrivial} if all its smooth fibres are isomorphic to each other.
$f$ is called {\it relatively minimal},
if there is no $(-1)$-curve contained in fibres of $f$.
Here a curve $C$ is called a {\it $(-k)$-curve}
if it is a smooth rational curve with self-intersection $C^2=-k$.
Note that by definition, $f$ is relatively minimal if $f$ is semi-stable.

Let $\omega_S$ (resp. $\omega_B$) be the canonical sheaf of $S$ (resp. $B$).
Denote by $\omega_{S/B}=\omega_S\otimes f^*\omega_B^{\vee}$ the
relative canonical sheaf of $f$.
let $b=g(B)$, $p_g=h^0(S,\,\omega_S)$, $q=h^0(S,\,\Omega_S^1)$,
$\chi(\mathcal O_S)=p_g-q+1$, and $\chit(X)$ the topological Euler characteristic of a variety $X$,
where $\Omega_S^1$ is the differential sheaf of $S$.
For a semi-stable family $f:\,S \to B$ of genus $g\geq 2$
with singular fibres $\Upsilon/\Delta$,
we consider the following relative invariants:
\begin{equation}\label{defofrelativeinv}
\left\{
\begin{aligned}
&\omega_{S/B}^2=\omega_S^2-8(g-1)(b-1),\\
&\delta_f=\chit(S)-4(g-1)(b-1)=\sum_{F\in \Upsilon}\delta(F),\\
&\deg f_*\Omega^1_{S/B}(\log\Upsilon)=\chi(\mathcal O_S)-(g-1)(b-1),
\end{aligned}\right.
\end{equation}
where
\begin{equation*}
\delta(F)=\chit(F)+(2g-2).
\end{equation*}
They satisfy the Noether's formula:
\begin{equation}\label{formulanoether}
12\deg f_*\Omega^1_{S/B}(\log\Upsilon)=\omega_{S/B}^2+\delta_f.
\end{equation}
These invariants are nonnegative. And $\deg f_*\Omega^1_{S/B}(\log\Upsilon)=0$
(equivalently, $\omega_{S/B}^2=0$)
if and only if $f$ is isotrivial.
Note that for a singular fibre $F$, $\delta(F)$ is also equal to the number of nodes contained $F$.
Hence $\delta_f=0$ iff $f$ is smooth, in which case,
$f$ is called a {\it Kodaira family} if moreover $f$ is non-isotrivial.

By contracting all $(-2)$-curves contained in fibres of a semi-stable family
$f:\,S \to B$, one gets a stable family $f^{\#}:\,S^{\#} \to B$,
$$\xymatrix{
  S \ar[rr] \ar[dr]_-{f}
                &  &    S^{\#}  \ar[dl]^-{f^{\#}}   \\
                & B                 }$$
In this case, of course, $S^{\#}$ is not necessarily smooth.
For any singular point $q$ of $S^{\#}$,
$(S^{\#},\,q)$ is a rational double point of type $A_{\lambda_q}$,
here $\lambda_q$ is the number of $(-2)$-curves in $S$ over $q$.

For each singular fibre $F$ of a semi-stable family $f:\, S \to B$ of genus $g\geq 2$,
we define $\delta_i(F)$ for $0\leq i \leq [g/2]$ in the following way.
A singular point $q$ of $F$ is said to be {\it of type $i$} ($0\leq i \leq [g/2]$)
if its partial normalization at $q$ consists of two connected components of arithmetic genera $i$
and $g-i$ for $i> 0$, and is connected for $i=0$.
Then $\delta_i(F)$ is the number of singular points in $F$ of type $i$.
We call also define $\delta_i(F)$ according to its stable model $F^{\#}\subseteq S^{\#}$.
To do this, first we similarly define singular points of type $i$ as before.
Around a singular point $q\in F^{\#}$ in $S^{\#}$,
locally $S^{\#}$ is of the form $xy=t^{m_q}$, where $t$ is a local coordinate of $B$ around $f^{\#}(q)$.
We call $m_q$ is the {\it multiplicity} of $q$.
Then $\delta_i(F)$ is the number of singular points of $F^{\#}$ of type $i$ counting multiplicity.
We remark that $(S^{\#},\,q)$ is a rational double point of type $A_{m_q-1}$,
if $m_q>1$ is the multiplicity of $q$.

Let $\Upsilon\to \Delta$
denote the singular fibres,
$\Upsilon_{c}\to \Delta_{c}$ denote those singular fibres with compact Jacobians, and
$\Upsilon_{nc}\triangleq \Upsilon\setminus \Upsilon_{c} \to \Delta_{nc}\triangleq\Delta\setminus\Delta_{c}$
correspond to singular fibres with non-compact Jacobians.
Define $\delta_h(F)=\sum\limits_{i=2}^{[g/2]} \delta_i(F)$, and
\begin{equation}\label{formulaofdelta_i}
\left\{
\begin{aligned}
&\delta_i(\Upsilon)=\sum_{F\in \Upsilon} \delta_i(F),\qquad
\delta_i(\Upsilon_c)=\sum_{F\in \Upsilon_c} \delta_i(F),\qquad
\delta_i(\Upsilon_{nc})=\sum_{F\in \Upsilon_{nc}} \delta_i(F).\\
&\delta_h(\Upsilon)=\sum_{i=2}^{[g/2]} \delta_i(\Upsilon),\qquad
\delta_h(\Upsilon_c)=\sum_{i=2}^{[g/2]} \delta_i(\Upsilon_c).
\end{aligned}
\right.
\end{equation}
Then
\begin{equation}\label{formulaofdelta_f}
\left\{
\begin{aligned}
\delta(F)&=\sum_{i=0}^{[g/2]}\delta_i(F)=\delta_0(F)+\delta_1(F)+\delta_h(F),\\
\delta_f&=\sum_{i=0}^{[g/2]}\delta_i(\Upsilon)=\delta_0(\Upsilon)+\delta_1(\Upsilon)+\delta_h(\Upsilon).
\end{aligned}
\right.
\end{equation}

Let $\calm_g$ (resp. $\ol \calm_g$) be the moduli space of smooth (resp. stable) complex curves of genus $g$.
By \cite{dm69},
the boundary $\ol \calm_g \setminus \calm_g$
is of codimensional one and has $[g/2]+1$ irreducible components,
saying, $\Delta_0,\,\Delta_1,\,\cdots,\,\Delta_{[g/2]}$.
The geometrical meaning of the index is as follows.
A general point of $\Delta_i$ ($i>0$) corresponds to a stable curve consisting of a curve of genus $i$
and a curve of genus $g-i$ joint at one point,
and a general point of $\Delta_0$ represents an irreducible stable curve with one node.
These boundaries define divisor classes
$\left\{\Delta_0,\,\Delta_1,\,\cdots,\,\Delta_{[g/2]}\right\}\in\Pic(\ol \calm_g)\otimes\mathbb Q$.
There is also a natural class $\lambda \in \Pic(\ol \calm_g)\otimes\mathbb Q$, called the Hodge class.
A non-isotrivial semi-stable family $f:\, S \to B$ of curves of genus $g\geq 2$ determines
a non-constant morphism $\varphi:\, B \to \ol M_g$.
Then one has (cf. \cite{dm69})
\begin{equation}\label{modulidelta}
\deg f_*\Omega^1_{S/B}(\log\Upsilon)=\deg \varphi^*(\lambda),\qquad
\delta_i(\Upsilon)=\deg \varphi^*(\Delta_i).
\end{equation}

\vspace{0.2cm}
For a singular points in a stable hyperelliptic curve $F^{\#}$,
we have a more detail description by using the induced double cover.
To see this, first note that $F^{\#}$ has a semi-stable model $F$
which is an admissible double cover (cf. \cite{ch88} or \cite{harrismumford82})
of a stable $(2g+2)$-pointed noded curve $\Gamma$ of arithmetic genus zero.
Let $\psi:\,F\to \Gamma$ be the covering map, and let $p$ be a singular point of $\Gamma$.
The complement of $p$ has two connected components $\Gamma'$ and $\Gamma''$,
so the set of marked points of $\Gamma$ breaks up into two subsets:
those lying on $\Gamma'$ and those lying on $\Gamma''$;
let $\alpha$ and $2g+2-\alpha\geq \alpha$ be the orders of these two subsets.
Following \cite[{P$_{467}$}]{ch88}, $\alpha$ is called the {\it index} of the point $p$.
Note that $\alpha\geq 2$.
If $p$ has odd index $\alpha=2k+1$, then $\psi$ must be branched at $p$,
and the unique singular point $q$ lying above $p$ is a singular point of type $k$.
Suppose that the index $\alpha=2k+2$ is even, then $\psi$ is unbranched at $p$,
so $\psi^{-1}(p)$ consists of two points $q'$ and $q''$,
and $\psi^{-1}(\Gamma')$ and $\psi^{-1}(\Gamma'')$ are semi-stable hyperelliptic curves of genera
$k$ and $g-k-1$, joint at $q'$ and $q''$.
In particular, both $q'$ and $q''$ are singular points of type $0$.

Let $f:\, S \to B$ be a semi-stable family of hyperelliptic curves,
and $F$ a singular fibre.
Let $\wt B\to B$ be a base change of degree $d$ and totally branched over $f(F)$,
$\tilde f:\,\wt S \to \wt B$ the corresponding semi-stable family,
and $\wt F$ the pre-image of $F$.
If $d$ is sufficiently large,
then we see that $\wt F$ is an admissible double cover
of a stable $(2g+2)$-pointed noded curve $\wt\Gamma$ of arithmetic genus zero.
Let $\xi_0(\wt F)$ equal to two times the number of singular points of $\wt \Gamma$ of index $2$,
and $\xi_j(\wt F)$ equal to the number of singular points of $\wt \Gamma$ of index $2j+2$
for $1\leq j\leq [(g-1)/2]$.
Define
$$\xi_j(F)=\frac{\xi_j(\wt F)}{d},\qquad \forall~0\leq j\leq [(g-1)/2].$$
It is clear that $\xi_j(F)$'s are independent on the choices of the base change,
and hence well-defined invariants.
And
$$\delta_0(F)=\xi_0(F)+2\sum_{j=1}^{[(g-1)/2]}\xi_j(F).$$
We also define
$$\xi_j(\Upsilon)=\sum_{F\in\Upsilon}\xi_j(F),\qquad \forall~0\leq j\leq [(g-1)/2].$$

Let $\calh_g\subseteq \calm_g$ (resp. $\ol \calh_g\subseteq \ol \calm_g$)
be the moduli space of smooth (resp. stable)
hyperelliptic complex curves of genus $g$.
The above discussion shows that the intersection of $\Delta_i$ and $\ol \calh_g$ ($i>0$)
is still an irreducible divisor (see \cite{ch88} for more details).
By abuse of notations, we still denote it by $\Delta_i$.
The intersection of $\Delta_0$ and $\ol \calh_g$, however, is reducible.
Let $\Xi_j$ be the locus of all curves in $\ol \calh_g$ such that
the corresponding marked pointed noded curve $\Gamma$ of arithmetic genus zero
has a singular point of index $2j+2$.
Then (cf. \cite{ch88})
$$\Delta_0\cap \ol \calh_g=\Xi_0\cup \Xi_1\cup\cdots\cup \Xi_{[(g-1)/2]}\,.$$
A general point in $\Xi_0$ corresponds to an irreducible stable hyperelliptic curve with one node.
A general point of $\Xi_j$ ($1\leq j\leq [(g-1)/2]$) represents a stable curve
consisting of a  hyperelliptic curve of genus $j$
and a hyperelliptic curve of genus $g-j-1$ joint at two points.
As divisors (cf. \cite{ch88}),
$$h^*\left(\Delta_0\right)=\Xi_0+2\sum_{j=1}^{[(g-1)/2]}\Xi_j\,,\qquad
\text{where~$h:\,\ol \calh_g\hookrightarrow \ol \calm_g$~is the embedding.}$$
Hence if $f:\,S\to B$ is a non-isotrivial
semi-stable family of hyperelliptic curves of genus $g$ with singular locus
$\Upsilon/\Delta$, and $\varphi:\, B \to \ol \calh_g$ the induced map,
then
\begin{equation}\label{relationdeltaxi}
\left\{
\begin{aligned}
\xi_j(\Upsilon)   &=\deg \varphi^*(\Xi_j),&&\forall~0\leq j\leq [(g-1)/2];\\
\delta_0(\Upsilon)&=\deg \varphi^*(\Xi_0)+2\sum_{j=1}^{[(g-1)/2]}\deg \varphi^*(\Xi_j);&&\\
\delta_i(\Upsilon)&=\deg \varphi^*(\Delta_i),&\quad&\forall~1\leq i\leq [g/2].
\end{aligned}\right.
\end{equation}

\section{Bounds of $\omega_{S/B}^2$ for arbitrary families}\label{sectionarbitrary}
\subsection{Lower bound of $\omega_{S/B}^2$}
The subsection aims to prove Theorem \ref{thmlower1}.
It follows directly from Moriwaki's sharp slope inequality (cf. \cite[Theorem D]{moriwaki98})
together with Noether's formula \eqref{formulanoether}.
\begin{proof}[Proof of Theorem {\rm\ref{thmlower1}}]~
From \cite[Theorem D]{moriwaki98}, \eqref{formulanoether} and \eqref{formulaofdelta_f}, it follows that
$$\begin{aligned}
(8g+4)\deg \left(f_*\Omega^1_{S/B}(\Upsilon)\right)&\,\geq\,
g\delta_0(\Upsilon)+\sum_{i=1}^{[g/2]}4i(g-i)\delta_i(\Upsilon)\\
&=\,g\delta_f+\sum_{i=1}^{[g/2]}\big(4i(g-i)-g\big)\delta_i(\Upsilon)\\
&=\,g\left(12\deg \left(f_*\Omega^1_{S/B}(\Upsilon)\right)-\omega_{S/B}^2\right)
+\sum_{i=1}^{[g/2]}\big(4i(g-i)-g\big)\delta_i(\Upsilon).
\end{aligned}$$
Hence
\begin{eqnarray*}
\omega_{S/B}^2&\geq&\frac{4(g-1)}{g}\cdot\deg \left(f_*\Omega^1_{S/B}(\Upsilon)\right)
+\sum_{i=1}^{[g/2]}\frac{4i(g-i)-g}{g}\delta_i(\Upsilon)\\
&\geq&\frac{4(g-1)}{g}\cdot\deg \left(f_*\Omega^1_{S/B}(\Upsilon)\right)
+\frac{3g-4}{g}\delta_1(\Upsilon)+\frac{7g-16}{g}\delta_h(\Upsilon).
\end{eqnarray*}
\end{proof}

\subsection{Upper bound of $\omega_{S/B}^2$}
The purpose of this subsection is to prove Theorem \ref{thmupper}.
The proof is based on a generalized Miyaoka-Yau's inequality (see Theorem \ref{theoreminpfupper1} below).
What we do is to take a suitable base change and to
choose suitable components contained in singular fibres (but not the entire singular fibres).

First we recall the generalized Miyaoka-Yao's theorem (cf. \cite{miyaoka84}).

Let $X_x$ the germ of a quotient singularity of $(\mathbb C^2/G_x)_0$ (in the analytic sense),
where $G_x$ is a finite subgroup of ${\rm GL}(2,\mathbb C)$ with the origin $0$ being its unique fixed point.
Let $\ol X_E$ be the minimal resolution of $X_x$ and $E$ the exceptional divisor
(= the inverse image of $x$). Let
\begin{equation}\label{eqninpfupper1}
v(x)\triangleq \chit(E)-\frac{1}{|G_x|}.
\end{equation}
\begin{theorem}[Miyaoka\,\,{\cite[Theorem\,1.1]{miyaoka84}}]\label{theoreminpfupper1}
Let $X^{\#}$ be a projective surface with only quotient singularities, and $\Lambda$ the
singular locus of $X^{\#}$. Let $D$ be a reduced normal crossing curve
which lies on the smooth part of $X^{\#}$. Let $\ol X$
be the minimal resolution of $X^{\#}$ and $E\subseteq \ol X$
the inverse image of $\Lambda$ (with reduced structure).
Assume the negative part in the Zariski decomposition of $\omega_{\ol X} + D + E$ has the form
$N + N'$ such that ${\rm supp\,}N$ is disjoint with $E$ and ${\rm supp\,}N' \subseteq E$.
Then
$$\sum_{x\in \Lambda}v(x) \leq \chit(\ol X)-\chit(D)-
\frac13(\omega_{\ol X}+E+ D)^2 +\frac13 (N')^2 +\frac1{12} N^2.$$
\end{theorem}

If $X^{\#}$ contains at most rational singularities of type $A$ (cf. \cite[\S\,III-3]{bhpv}),
$\ol X$ is minimal and of general type, and $D$ is composed of some disjoint smooth elliptic curves,
then $\chit(D)=0$ and the negative part in the Zariski decomposition of $\omega_{\ol X}+D+E$ is just $E$,
which is some chains of $(-2)$-curves.
Hence in this case, $$\left(\omega_{\ol X}+D+E\right)^2=(\omega_{\ol X}+D)^2+E^2.$$
Note that for a singularity $x$ of type $A_k$,
the invariant $v(x)$ defined in \eqref{eqninpfupper1} is equal to $(k+1)-\frac{1}{k+1}$.
Therefore we get
\begin{theorem}\label{theoreminpfupper2}
Let conditions be the same as that of Theorem $\ref{theoreminpfupper1}$.
Assume that each point $x\in \Lambda$
is a quotient singularities of type $A_{k_x}$,
$\ol X$ is minimal and of general type, and $D$ is composed of some disjoint smooth elliptic curves.
Then
\begin{equation*}
\sum_{x\in \Lambda}\left((k_x+1)-\frac{1}{k_x+1}\right)
\leq \chit(\ol X)-\frac13(\omega_{\ol X}+ D)^2.
\end{equation*}
\end{theorem}

We are now able to prove Theorem \ref{thmupper}.
\begin{proof}[Proof of Theorem {\rm\ref{thmupper}}]~
We use the notations introduced in Sections \ref{sectionintroduction} and \ref{sectionprelim}.

Consider first the case $g=2$.
In this case, $f$ is a semi-stable hyperelliptic family,
whence $\Upsilon\neq \emptyset$;
otherwise, from \cite[Proposition 4.7]{ch88}, it follows that
$\deg f_*\big(\Omega_{S/B}^1(\log\Upsilon)\big)=0$,
which is impossible, since $f$ is assumed to be non-isotrivial.
Note that for a singular fibre $F \in \Upsilon_c$,
$\delta_1(F)\geq 1$.
Hence $\delta_1(\Upsilon_c)\geq \#(\Delta_c)$.
Therefore when $g=2$,
our theorem is a direct consequence of the strict canonical class inequality (cf. \cite{tan95}):
\begin{eqnarray*}
K_f^2&<&(2g-2)\cdot \deg \big(\Omega^1_B(\log \Delta)\big)\\
&=&2\deg \big(\Omega^1_B(\log \Delta_{nc})\big)+2\cdot\#(\Delta_c)\\
&\leq& 2\deg \big(\Omega^1_B(\log \Delta_{nc})\big)+2\delta_1(\Upsilon_c).
\end{eqnarray*}

In the rest part of the proof, we assume $g\geq 3$.

Let $s_c=\#\Sigma_{c}$, $s_{nc}=\#\Sigma_{nc}$.
For any $p\in \Delta$, let $F_p=f^{-1}(p)$,
and
\begin{eqnarray}
\hspace{-0.6cm}E_p&\hspace{-0.2cm}=&\hspace{-0.2cm}
\left\{\sum_{j} E_{p,j}~\,\Big|~ E_{p,j}\subseteq F_p \text{~is a~}
(-2)\text{-curve }\right\},\quad \forall~p\in\Delta;\label{miyaoka13}\\[.1cm]
\hspace{-0.6cm}D_p&\hspace{-0.2cm}=&\hspace{-0.2cm}\left\{\sum_{j} D_{p,j}~\Bigg|~
\begin{aligned}
&D_{p,j}\subseteq F_p \text{~is a smooth elliptic curve such that}\\
&D_{p,j}\cdot E_p=0, \text{~and~}D_{p,j}^2=-1
\end{aligned}\right\}, \quad\forall~p\in\Delta_c.~~\label{miyaoka14}
\end{eqnarray}

Let $S \to S^{\#}$ be the contraction of $\sum\limits_{p\in\Delta} E_p\subseteq S$,
and $f^{\#}:\,S^{\#} \to B$ the induced morphism.
Then $f^{\#}:\,S^{\#} \to B$ is nothing but the stable model of $f$.
Note also that, for any $p\in\Delta_c$, the image of $D_p$ on $S^{\#}$ lies on
the smooth part of ${S^{\#}}$,
which we still denote by $D_p$.
For any singular point $q$ of $S^{\#}$,
$(S^{\#},\,q)$ is a rational double point of type $A_{\lambda_q}$,
here $\lambda_q$ is the number of $(-2)$-curves in $S$ over $q$.
For convenience, we also denote by $q$ the singular point of the fibres on the smooth
part of $S^{\#}$, in which case, $\lambda_q=0$.
So a singular point $(S^{\#},\,q)$ of type $A_0$ is understood as
a node of the fibres but a smooth point of $S^{\#}$.
Let $F^{\#}_p$ be the image of $F_p$ on $S^{\#}$ for $p\in \Delta$,
and
$$\Upsilon^{\#}=\sum_{p\in \Delta} F^{\#}_p,\qquad
\Upsilon_{c}^{\#}=\sum_{p\in \Delta_c} F^{\#}_p,\qquad
\Upsilon_{nc}^{\#}=\sum_{p\in \Delta_{nc}} F^{\#}_p.$$
Then
\begin{equation}\label{miyaoka11}
\delta(F_p)=\sum_{i=0}^{[g/2]}\delta_i(F_p)
=\sum_{q\in F^{\#}_p}(\lambda_q+1),\qquad \forall~p\in \Delta.
\end{equation}
We claim
\begin{claim}\label{e+Dgeqdelta1}
For each $p\in \Delta_c$,
$D_p$ is smooth (not necessary irreducible), and
\begin{equation}\label{miyaoka10}
\sum_{q\in F_{p}^{\#}\atop{\lambda_q>0}}(\lambda_q+1)+|D_p| \geq \delta_1(F_p),
\end{equation}
where~$|D_p|$~is the number of irreducible components of $D_p$.
\end{claim}
We leave the proof of the claim at the end of the section.

Let $\phi:\,\ol B\to B$ be a cover of $B$, such that $\deg \phi=de$ and $\phi$ ramifies
uniformly over $\Delta_{nc}$ with ramification index equaling to $e$.
By Kodaira-Parshin construction, such a cover exists for all $e$ if
$b=g(B) > 0$ and for odd $e\geq 3$ if $b = 0$ (cf. \cite{vojta88} or \cite{tan95}).
Let $\bar b=g(\ol B)$ be the genus of $\ol B$.
According to Hurwitz formula, we get
\begin{equation}\label{miyaoka2}
2(\bar b-1)=de\cdot\left(2(b-1)+\frac{e-1}{e}\cdot s_{nc}\right).
\end{equation}

Let $\ol {S^{\#}}=\ol B \times_B S^{\#}$ be the fibre-product,
$\ol S \to \ol {S^{\#}}$ the minimal resolution of singularities.
We have the following commutative diagram:
\begin{center}\mbox{}
\xymatrix{
\ol S \ar@/_7mm/"3,1"_{\bar f} \ar[d] \ar[rr]^{\ol\Phi} && S \ar[d] \ar@/^7mm/"3,3"^f \\
\ol{ S^{\#}} \ar[d] \ar[rr]^{\Phi^{\#}} && S^{\#} \ar[d]\\
\ol{ B} \ar[rr]^{\phi} && B}
\end{center}

For $p\in\Delta_c$,
the inverse image of a singular point $(S^{\#},\,q)$ of type $A_{\lambda_q}$ with $q\in F^{\#}_p$
is $de$ singular points of the same type $A_{\lambda_q}$ in $\ol{ S^{\#}}$.
For $p\in\Delta_{nc}$,
the inverse image of a singular point $(S^{\#},\,q)$ of type $A_{\lambda_q}$ with $q\in F^{\#}_p$
is $d$ singular points of type $A_{e(\lambda_q+1)-1}$ in $\ol{ S^{\#}}$.
Let $$D=\sum\limits_{p\in\Delta_c}\big(\Phi^{\#}\big)^{-1}(D_p).$$
Since $\phi$ is unbranched over $\Delta_c$,
$D$ is smooth and lies on the smooth part of $\ol {S^{\#}}$,
and the number of irreducible components in $D$ is
$$|D|=de\cdot \sum_{p\in\Delta_c}|D_p|.$$
Hence
\begin{equation}\label{miyaoka4}
2\,\omega_{\ol S}\cdot D+D^2=|D|=de\cdot\sum_{p\in\Delta_c}|D_p|.
\end{equation}
Because $f$ is semi-stable, $\bar f:\,\ol S \to \ol B$ is also semi-stable, and
\begin{equation}\label{miyaoka5}
\delta_{\bar f}=de\cdot \delta_f,\qquad \omega_{\ol S/\ol B}^2=de\cdot \omega_{S/B}^2.
\end{equation}
It is not difficult to see that $\ol S$ is minimal and of general type if $\bar b=g(\ol B) \geq 1$,
which is satisfied when $de$ is large enough.
Hence applying Theorem \ref{theoreminpfupper2} to the case
by setting $X^{\#}=\ol{ S^{\#}}$, $\ol X=\ol S$, and $D$ as above,
we get
\begin{eqnarray}
&&de\cdot \sum_{q\in \Upsilon^{\#}_{c}}\left((\lambda_q+1)-\frac{1}{\lambda_q+1}\right)
+d\cdot\sum_{q\in \Upsilon^{\#}_{nc}}\left(e(\lambda_q+1)-\frac{1}{e(\lambda_q+1)}\right)\nonumber\\
&\leq& \chit(\ol S)-\frac13\left(\omega_{\ol S}+D\right)^2\nonumber\\
&=& de\cdot \left(\delta_f-\frac13\omega_{S/B}^2-
\frac13\sum_{p\in\Delta_c}|D_p|\right)
+\frac13de\cdot(2g-2)\left(2b-2+\frac{e-1}{e}\cdot s_{nc}\right).\qquad\quad\label{miyaoka3}
\end{eqnarray}
We use \eqref{defofrelativeinv}, \eqref{miyaoka2},
\eqref{miyaoka4} and \eqref{miyaoka5} in the last step above.
Note that
\begin{eqnarray}
\delta_f=\sum_{p\in\Delta} \delta(F_p)
&=&\sum_{q\in \Upsilon_{c}^{\#}}(\lambda_q+1)
+\sum_{q\in \Upsilon_{nc}^{\#}}(\lambda_q+1),\label{miyaoka8}\\[0.1cm]
\sum_{q\in \Upsilon_{c}^{\#}}\frac{3}{\lambda_q+1}
&=& \sum_{q\in \Upsilon_{c}^{\#}\atop{\lambda_q=0}}\frac{3}{\lambda_q+1}+
\sum_{q\in \Upsilon_{c}^{\#}\atop{\lambda_q>0}}\frac{3}{\lambda_q+1}    \nonumber\\
&\leq & \sum_{q\in \Upsilon_{c}^{\#}\atop{\lambda_q=0}} 3(\lambda_q+1)+
\sum_{q\in \Upsilon_{c}^{\#}\atop{\lambda_q>0}}\frac{3(\lambda_q+1)}{4} \nonumber\\
&=&\sum_{q\in \Upsilon_{c}^{\#}}3(\lambda_q+1)-
\sum_{q\in \Upsilon_{c}^{\#}\atop{\lambda_q>0}}\frac{9(\lambda_q+1)}{4} \nonumber\\
&\leq &\sum_{q\in \Upsilon_{c}^{\#}}3(\lambda_q+1)-
\sum_{q\in \Upsilon_{c}^{\#}\atop{\lambda_q>0}}(\lambda_q+1).\label{miyaoka1}
\end{eqnarray}
Combining \eqref{miyaoka1} with \eqref{miyaoka11} and \eqref{miyaoka10}, one gets
\begin{eqnarray}
\sum_{q\in \Upsilon_{c}^{\#}}\frac{3}{\lambda_q+1}-\sum_{p\in\Delta_c}|D_p|
&\leq& \sum_{q\in \Upsilon_{c}^{\#}}3(\lambda_q+1)-
\sum_{p\in\Delta_c}\left(\sum_{q\in F_{p}^{\#}\atop{\lambda_q>0}}(\lambda_q+1)+|D_p|\right)\nonumber\\
&\leq&\sum_{q\in \Upsilon_{c}^{\#}}3(\lambda_q+1)-\sum_{p\in\Delta_c}\delta_1(F_p)
=2\delta_1(\Upsilon_{c})+2\delta_h(\Upsilon_{c}).\qquad\qquad\label{miyaoka7}
\end{eqnarray}
Therefore by \eqref{miyaoka3}, \eqref{miyaoka8} and \eqref{miyaoka7}, we get
$$\begin{aligned}
\omega_{S/B}^2&\leq
(2g-2)(2b-2+s_{nc})
+\sum_{q\in \Upsilon_{c}^{\#}}\frac{3}{\lambda_q+1}-\sum_{p\in\Delta_c}|D_p|
+\frac{1}{e^2}\cdot\sum_{q\in \Upsilon_{nc}^{\#}}\frac{3}{\lambda_q+1}-\frac{(2g-2)s_{nc}}{e}\\[.1cm]
&\leq(2g-2)(2b-2+s_{nc})+2\delta_1(\Upsilon_{c})+2\delta_h(\Upsilon_{c})
+\frac{1}{e^2}\cdot\sum_{q\in \Upsilon_{nc}^{\#}}\frac{3}{\lambda_q+1}-\frac{(2g-2)s_{nc}}{e}.
\end{aligned}$$
Letting $e$ tend to infinity, we get the required inequality \eqref{eqnupper}.
If $s_{nc}>0$,
then letting $e$ be large enough, one has
$$\frac{1}{e^2}\cdot\sum_{q\in \Upsilon_{nc}^{\#}}\frac{3}{\lambda_q+1}-\frac{(2g-2)s_{nc}}{e}<0.$$
Hence if $\Delta_{nc}\neq \emptyset$, then the inequality \eqref{eqnupper} is strict.
Finally, if $\Delta=\emptyset$, then $f$ is a Kodaira family, and
$$\deg\left(\Omega^1_{B}(\log\Delta_{nc})\right)=2b-2,\qquad
\delta_1(\Upsilon_{c})=\delta_h(\Upsilon_{c})=0,$$
so it follows from \cite[Corollary\,0.6]{liukefeng96} that \eqref{eqnupper} is also strict.
\end{proof}
\begin{remark}\label{miyaoka12}
If \eqref{eqnupper} is indeed an equality, i.e.,
\begin{equation}\label{miyaoka9}
\omega_{S/B}^2 = (2g-2)\cdot \deg\left(\Omega^1_{B}(\log\Delta_{nc})\right)
+2\delta_1(\Upsilon_{c})+3\delta_h(\Upsilon_{c}),
\end{equation}
then $\Delta_{nc}=\emptyset$; $E_p=\emptyset$ for $p\in\Delta_c$ by \eqref{miyaoka1};
and $\delta_1(\Upsilon_{c})=\sum\limits_{p\in\Delta_c}|D_p|$ by \eqref{miyaoka10}.
Hence \eqref{miyaoka9} is equivalent to
$$c_1^2\left(\Omega_S^1\bigg(\log \Big(\sum\limits_{p\in\Delta_c}D_p\Big)\bigg)\right)
=3c_2\left(\Omega_S^1\bigg(\log \Big(\sum\limits_{p\in\Delta_c}D_p\Big)\bigg)\right).$$
It follows that $S\setminus\left(\bigcup\limits_{p\in\Delta_c}D_{p}\right)$
is a ball quotient by \cite{kobayashi} or \cite{mok12}.
\end{remark}

\begin{proof}[Proof of Claim {\rm\ref{e+Dgeqdelta1}}]~
First we prove that $D_p$ is smooth. Assume that $D_p$ is not smooth.
As each irreducible component of $D_p$ is smooth by definition,
there are two irreducible components $D_{p,1}$, $D_{p,2}$
contained in $D_p$ such that $D_{p,1} \cap D_{p,2}\neq \emptyset$.
Since $D_{p,j}^2=-1$, each irreducible component $D_{p,j}$ of $D_p$ intersects $F_p- D_{p,j}$
in exactly one point.
So $F_p=D_p=D_{p,1}+D_{p,2}$ with $\omega_S\cdot D_{p,1}=\omega_S\cdot D_{p,2}=1$,
and hence $2g-2=\omega_S\cdot F_p=2$.
It is impossible, since $g\geq 3$ by our assumption.

It remains to prove \eqref{miyaoka10}.

For this purpose, we make use of the stable model $F_p^{\#}$ of $F_p$.
As $F_p$ has compact Jacobian, $F_p^{\#}$ is tree of smooth curves.
Hence a singular point in $F_p^{\#}$ of type $1$ is a point $q$ such that
the partial normalization at $q$ consists of a smooth elliptic curve $D_q$ and a curve of
arithmetic genera $g-1$.
And $\delta_1(F_p)$ is by definition is the number of singular points
in $F_p^{\#}$ of type $1$ counting multiplicity.

Let $m_q$ be the multiplicity of a singular point $q\in F^{\#}_p$ of type $1$.
If $m_q=1$, then $(S^{\#},\,q)$ is smooth at $q$, hence the inverse image of $D_q$ in $S$
is a smooth elliptic curve with self-intersection equal to $-1$
and does not intersect $E_p$.
It means that the inverse image of $D_q$ is contained in $D_p$.
If $m_q>1$, then $(S^{\#},\,q)$ is a rational double point of type $A_{m_q-1}$,
hence $\lambda_q=m_q-1>0$.
Therefore, \eqref{miyaoka10} follows immediately.
\end{proof}

\section{Lower bound of $\omega_{S/B}^2$ for hyperelliptic families}\label{sectionlowerhyper}
The section aims to prove Theorem \ref{thmlower}.
So we always assume that $f:\,S \to B$ is a non-isotrivial semi-stable family of
hyperelliptic curves of genus $g\geq2$ with relative irregularity $q_f=q(S)-g(B)$.

When $q_f=0$, it is a direct consequence of the Noether's formula and the following formula
given in \cite[Proposition 4.7]{ch88}:
\begin{equation}\label{formulaofdegomega}
\begin{aligned}
\deg \left(f_*\Omega^1_{S/B}(\Upsilon)\right)\,=&~\frac{g}{4(2g+1)}\xi_0(\Upsilon)\\
&~+\sum_{i=1}^{[g/2]}\frac{i(g-i)}{2g+1}\delta_i(\Upsilon)
+\sum_{j=1}^{[(g-1)/2]}\frac{(j+1)(g-j)}{2(2g+1)}\xi_j(\Upsilon).
\end{aligned}
\end{equation}
When $q_f>0$,
the proof starts from the observation that the double cover $\pi:\, S \to S/\langle\sigma\rangle$
is fibred, where $\sigma$ is the involution on $S$ induced by the hyperelliptic involution on fibres of $f$.
From this it follows a restriction on those invariants $\delta_i(\Upsilon)$'s and $\xi_j(\Upsilon)$'s
(cf. Proposition\,\ref{boundofxi_0prop}).
And the proof of Theorem \ref{thmlower} is completed in Section \ref{subsectpfthmlower}
by combining this with \eqref{formulaofdegomega}.

\subsection{Proof of \eqref{eqnlower} for $q_f=0$}
By \eqref{formulaofdelta_f} and \eqref{relationdeltaxi}, one has
\begin{equation}\label{formulaofdelta_f''}
\delta_f=\xi_0(\Upsilon)+\sum_{i=1}^{[g/2]}\delta_i(\Upsilon)
+2\sum_{j=1}^{[(g-1)/2]}\xi_j(\Upsilon).
\end{equation}
From the above equation together with Noether's formula and \eqref{formulaofdegomega}, it follows that
\begin{equation}\label{formulaofomega^2}
\begin{aligned}
\omega^2_{S/B}\,=&~\frac{g-1}{2g+1}\xi_0(\Upsilon)\\
&~+\sum_{i=1}^{[g/2]}\left(\frac{12i(g-i)}{2g+1}-1\right)\delta_i(\Upsilon)
+\sum_{j=1}^{[(g-1)/2]}\left(\frac{6(j+1)(g-j)}{2g+1}-2\right)\xi_j(\Upsilon).
\end{aligned}
\end{equation}
Hence
\begin{eqnarray*}
&&     \omega_{S/B}^2-\frac{4(g-1)}{g}
       \cdot\deg \left(f_*\Omega^1_{S/B}(\Upsilon)\right)\\
& =  & \sum_{i=1}^{[g/2]}\frac{4i(g-i)-g}{g}\delta_i(\Upsilon)
       +\sum_{j=1}^{[(g-1)/2]}\frac{2(j+1)(g-j)-2g}{g}\xi_j(\Upsilon)\\
&\geq&
\left\{
\begin{aligned}
&\frac{3g-4}{g} \delta_1(\Upsilon)+\frac{7g-16}{g} \delta_h(\Upsilon), &\qquad&
\text{if~} \Delta_{nc}\neq \emptyset;\\
&\sum_{i=1}^{[g/2]}\frac{4i(g-i)-g}{g}\delta_i(\Upsilon), &&
\text{if~}\Delta_{nc}= \emptyset.
\end{aligned}\right.
\end{eqnarray*}
Hence \eqref{eqnlower} is proved for $q_f=0$.
\qed

\subsection{Hyperelliptic family with positive relative irregularity}
The main purpose of this subsection is to prove the following
proposition for a semi-stable hyperelliptic family with positive relative irregularity.
\begin{proposition}\label{boundofxi_0prop}
Let $f:\,S\to B$ be a non-isotrivial semi-stable family of
hyperelliptic curves of genus $g$ with singular locus $\Upsilon/\Delta$.
Let $\delta_i(\Upsilon)$'s and $\xi_j(\Upsilon)$'s be defined in Section \ref{sectionprelim}.
Assume that the relative irregularity $q_f=q(S)-g(B)>0$. Then
\begin{equation}\label{boundofxi_0eqn}
\begin{aligned}
&\sum_{i= q_f}^{[g/2]}\frac{(2i+1)(2g+1-2i)}{g+1}\delta_i(\Upsilon)
+\sum_{j= q_f}^{[(g-1)/2]}\frac{2(j+1)(g-j)}{g+1}\xi_j(\Upsilon)\\
\geq~&~ \xi_0(\Upsilon)+\sum_{i= 1}^{q_f-1}4i(2i+1)\delta_i(\Upsilon)
+\sum_{j= 1}^{q_f-1}2(j+1)(2j+1)\xi_j(\Upsilon).
\end{aligned}
\end{equation}
\end{proposition}

As said before, the key point is the observation that
the induced double cover $\pi:\, S \to S/\langle\sigma\rangle$ is fibred.
To be more precise,
let $f:\,S\to B$ be as in the above proposition,
and $f^{\#}:\,S^{\#} \to B$ the stable model.
The hyperelliptic involution induces a double cover
$\pi^{\#}:\,S^{\#}\to Y^{\#}$.
By resolving the singular points, one gets a double cover
$\tilde\pi:\,\wt S \to \wt Y$ between smooth surfaces
with smooth branched divisor $\wt R \subseteq \wt Y$.
Let $\tilde f:\,\wt S \to B$ and $\tilde h:\, \wt Y \to B$ be the induced morphism.
\renewcommand{\thefigure}{\arabic{section}.\arabic{subsection}-\arabic{figure}}
\begin{figure}[H]
\begin{center}\mbox{}
\xymatrix{
 \wt S \ar@{->}[rr]^-{\tilde\pi} \ar[d] && \wt Y  \ar@{->}[d] \\
 S^{\#} \ar@{->}[rr]^-{\pi^{\#}} \ar[rd]_-{f^{\#}}   && Y^{\#}\ar[ld]^-{h^{\#}}\\
 &B&
}\vspace{-0.3cm}
\end{center}
\caption{Hyperelliptic involution.\label{hyperellipticdiagram}}
\end{figure}
We would like to show that $\tilde\pi:\,\wt S \to \wt Y$ is fibred.
First we recall the following definition.
\begin{definition}[\cite{khashin83}]\label{definitonfibred}
A double cover $\pi:\, X\to Z$ of smooth projective surfaces with branched divisor
$R \subseteq Z$ is called fibred if there exist a double cover $\pi':\, C \to D$ of smooth projective
curves and morphisms $\varpi:\, X\to C $ and $\epsilon:\, Z \to D$
with connected fibres, such that the diagram
\begin{center}\mbox{}
\xymatrix{
 X \ar@{->}[rr]^-{\pi}\ar@{->}[d]_-{\varpi} && Z \ar@{->}[d]^-{\epsilon}\\
 C \ar@{->}[rr]^-{\pi'}                  && D
}
\end{center}
is commutative, $R$ is contained in the fibres of $\epsilon$, and $q(X)-q(Z) = g(C)-g(D)$,
where $g(C)$ (resp. $g(D)$) is the genus of $C$ (resp. $D$),
$q(X)=\dim H^0(X,\,\Omega^1_X)$, and $q(Z)=\dim H^0(Z,\,\Omega^1_Z)$.
\end{definition}

The next theorem is proved in \cite{khashin83}.
For readers' convenience, we reprove it here.
\begin{theorem}[{\cite[Theorem 1]{khashin83}}]\label{qX>qYfibred}
If the geometrical genus $p_g(Z) =\dim H^0(Z,\,\Omega^2_Z)= 0$. Then any double cover
$\pi: X\to Z$ with smooth branched divisor $R \subseteq Z$ and with $q(X) > q(Z)$ is fibred.
\end{theorem}
\begin{proof}
Note that the Galois group ${\rm Gal}(X/Z)\cong \mathbb Z_2$ has a natural action on
$H^0(X,\Omega^1_X)$.
Let $$H^0(X,\Omega^1_X)=H^0(X,\Omega^1_X)_1\oplus H^0(X,\Omega^1_X)_{-1}$$
be the eigenspace decomposition.
Then
$$H^0(X,\Omega^1_X)_1=\pi^*H^0(Z,\Omega^1_Z), \quad\text{and}\quad
k\triangleq\dim H^0(X,\Omega^1_X)_{-1}=q(X)-q(Z).$$
Let $\omega_1,\cdots,\omega_k$ be a basis of $H^0(X,\Omega^1_X)_{-1}$.
First we prove that there exists a morphism \mbox{$\varpi:\,X \to C$}
to a curve $C$ with connected fibres, such that there exist
$\alpha_1,\cdots,\alpha_k \in H^0(C,\Omega^1_C)$ satisfying
\begin{equation}\label{omega=pullalpha}
\omega_i=\varpi^*(\alpha_i), \qquad \forall\, 1\leq i\leq k.
\end{equation}

If $k\geq 2$, then
$$\omega_i\wedge\omega_j \in \wedge^2 H^0(X,\Omega^1_X)_{-1} \subseteq  H^0(X,\Omega^2_X)$$
 is invariant under the action
of the Galois group for any $i\neq j$.
Hence it belongs to
$$H^0(X,\Omega^2_X)_1=\pi^*\left(H^0(Z,\Omega^2_Z)\right),$$
which is zero by our assumption $p_g(Z)=\dim H^0(Z,\Omega^2_Z)=0$.
By \cite[\S IV-5]{bhpv}, there exists a morphism $\varpi:\,X \to C$ with connected fibres
such that \eqref{omega=pullalpha} holds.

For the case $k=1$, note that $\pi:\,X\to Z$ induces a surjective morphism
${\rm Alb}(\pi):\,{\rm Alb}(X)\to {\rm Alb}(Z)$ between abelian varieties as follows.
\begin{center}\mbox{}
  \xymatrix{
   X \ar[rr]^{\pi}\ar[d]_{{\rm Alb}_X}&& Z \ar[d]^{{\rm Alb}_Z}\\
   {\rm Alb}(X) \ar[rr]^{{\rm Alb}(\pi)}&& {\rm Alb}(Z)
   }
\end{center}
By assumption $$\dim {\rm Alb}(X)-\dim {\rm Alb}(Z)=q(X)-q(Z)=k=1.$$
According to the theory on abelian varieties (cf. \cite{mumford74}),
there exists a one-dimensional abelian variety (i.e., an elliptic curve) $C_0$
such that ${\rm Alb}(X)$ is isogenous to ${\rm Alb}(Z) \times C_0$, i.e.,
there exists a morphism $\varphi:\,{\rm Alb}(X) \to {\rm Alb}(Z) \times C_0$ with finite kernel.
Let ${\rm pr}:\,{\rm Alb}(Z) \times C_0 \to C_0$ be the projection, and
$$\varpi_0\triangleq{\rm pr}\circ \varphi\circ{\rm Alb}_X:~X \lra C_0$$
be the composition.
As ${\rm Alb}_X(X)$ generates ${\rm Alb}(X)$, $\varpi_0$ is surjective.
By Stein factorization (cf. \cite[\S\,III-11]{hartshorne}),
we get a morphism $\varpi:\,X \to C$ with connected fibres.
And \eqref{omega=pullalpha} is clearly satisfied.

Note that the property \eqref{omega=pullalpha} implies that the morphism $\varpi:\,X \to C$ is unique.
In particular, the Galois group ${\rm Gal}(X/Z)\cong \mathbb Z_2$ induces an automorphism group $G$ on $C$.
Let $D=C/G$, and $\pi':\,C \to D$ be the natural morphism.
Then by construction, there exists a morphism
$\epsilon:\,Z \to D$ such that $\epsilon\circ \pi=\pi'\circ \varpi$.

Let $\sigma$ be the non-identity element of ${\rm Gal}(X/Z)$.
Then the fixed locus Fix$(\sigma)$ of $\sigma$ is clearly contained in the fibres of $\varpi$.
So $R=\pi\big(\text{Fix}(\sigma)\big)$ is contained in the fibres of $\epsilon$.
By \eqref{omega=pullalpha}, one sees that
the eigenspace decomposition of $H^0(C,\Omega_C^1)$ with respect to the action of $G$ is
$$H^0(C,\Omega_C^1)= (\pi')^*H^0(D,\Omega_D^1)\oplus H^0(C,\Omega_C^1)_{-1}\,,$$
where $H^0(C,\Omega_C^1)_{-1}$ is generated by $\alpha_1$, $\cdots$, $\alpha_k$.
So
$$q(X)-q(Z)=\dim H^0(X,\Omega_X^1)_{-1}=\dim H^0(C,\Omega_C^1)_{-1} = g(C)-g(D).$$
\end{proof}

Coming back to our case.
Note that $q(\wt S)=q(S)$ and $q(\wt Y)=g(B)$.
If $q_f=q(S)-g(B)>0$,
it follows that $q(\wt S)>q(\wt Y)$.
As $\wt Y$ is a ruled surface, the geometric genus $p_g(\wt Y)=0$.
Hence by Theorem \ref{qX>qYfibred} above, we get
\begin{proposition}\label{fibredprop}
The double cover $\tilde\pi:\,\wt S \to \wt Y$ is fibred, i.e.,
there exist a double cover $\pi': B' \to D$ of smooth projective
curves and morphisms $\tilde f':\, \wt S\to B' $ and $\tilde h':\, \wt Y \to D$
with connected fibres, such that the diagram
\begin{center}\mbox{}
\xymatrix{
 \wt S \ar@{->}[rr]^-{\pi}\ar@{->}[d]_-{\tilde f'} && \wt Y \ar@{->}[d]^-{\tilde h'}\\
 B' \ar@{->}[rr]^-{\pi'}                  && D
}
\end{center}
is commutative, $\wt R$ is contained in the fibres of $\tilde h'$ and
\begin{equation}\label{q_f=gB'-gZ}
q_f=q(\wt S)-q(\wt Y) = g(B')-g(D).
\end{equation}
\end{proposition}

Our purpose is to prove Proposition \ref{boundofxi_0prop}.
Before going to the detailed proof, we show that the curve $D$ in the above proposition
is actually isomorphic to $\bbp^1$, whence $g(B')=q_f$ by \eqref{q_f=gB'-gZ}.

Let $\wt F$ and $\wt F'$ be any fibres of
$\tilde f:\,\wt S \to B$ and $\tilde f':\, \wt S \to B'$ respectively.
By restrictions, we get the following two morphisms:
\begin{equation}\label{deff_1'}
\left\{
\begin{aligned}
\tilde f'|_{\wt F}:&~ \wt F \lra B',\\
\tilde f|_{\wt F'}:&~ \wt F' \lra B.
\end{aligned}\right.
\end{equation}
It is clear that $\tilde f'|_{\wt F}$ and $\tilde f|_{\wt F'}$ are surjective and
have the same degree
\begin{equation}\label{degf_1'=GammaGamma}
d=\deg \left(\tilde f'|_{\wt F}\right)
=\deg \left(\tilde f|_{\wt F'}\right)=\wt F\cdot\wt F'.
\end{equation}

\begin{proposition}\label{q_f>0fibred}
Let $f:\,S\to B$ be a non-isotrivial
semi-stable family of hyperelliptic curves of genus $g\geq 2$,
and $d$ be defined in \eqref{degf_1'=GammaGamma}.
Then $d\geq 2$, $D\cong \bbp^1$, and
\begin{equation}\label{boundofq_f}
q_f=g(B') \leq \frac{g-1}{d}+1.
\end{equation}
\end{proposition}
\begin{proof}
If $d=1$, then it follows that $\wt S$ is birational to $\wt F \times B'$.
It is a contraction, since $f$ is non-isotrivial.
So $d\geq 2$.

As $\tilde h:\, \wt Y \to B$ is the ruling of the ruled surface $\wt Y$,
a general fibre $\wt \Gamma$ of $\tilde h$ is isomorphic to $\bbp^1$.
By the discussion above, $\wt \Gamma$ is mapped surjective to $D$ by $\tilde h'$.
Hence $D\cong \bbp^1$, i.e., $g(D)=0$.
By \eqref{q_f=gB'-gZ}, $g(B')=q_f$.
According to Hurwitz formula for algebraic curves, we get
$$2g-2=2g(\wt F)-2 \geq d\cdot (2g(B')-2) = 2d\cdot (q_f-1).$$
So \eqref{boundofq_f} is proved.
\end{proof}

\begin{remark}\label{boundofq_fremark}
Let $f:\,S\to B$ be as in the above proposition with $b=g(B)\geq 1$.
It is not difficult to show that $d=\deg \left(\tilde f'|_{\wt F}\right)$
is nothing but the degree of the Albanese map $S\to {\rm Alb\,}(S)$.
Xiao (\cite{xiao92-0}) proved that if $$q_f=\frac{g-1}{d}+1,$$ then $f$ is isotrivial.
\end{remark}

\begin{proof}[Proof of Proposition {\rm\ref{boundofxi_0prop}}]~
In order to prove \eqref{boundofxi_0eqn},
we may limit ourselves to the family whose stable model $f^{\#}:\,S^{\#} \to B$
comes from an admissible double cover (cf. \cite{ch88} or \cite{harrismumford82});
that is, a double cover of a family $h^{\#}:\,Y^{\#} \to B$ of stable $(2g+2)$-pointed noded curves of
arithmetic genus zero, branched along the $2g+2$ disjoint sections $\sigma_i$ of $h^{\#}$
and possibly at some of the nodes of fibres of $h^{\#}$.
Actually, for any semi-stable family of hyperelliptic curves over a curve, we may get a family of
admissible covers by base change and blowing-ups of singular points in the fibres.
These operations have the effect of multiplying all the invariants $\delta_i(\Upsilon)$'s and
$\xi_j(\Upsilon)$'s by the same constant,
and the relative irregularity $q_f$ is non-decreasing under these operation.

Let $\Lambda=\{p_i\}$ be the set of points of $Y^{\#}$ which are nodes of their fibres.
Following \cite[{P$_{470}$}]{ch88}, if the local equation of $Y^{\#}$ at $p_i$ is $xy=t^{m_i}$,
then we say that $p_i$ has multiplicity $m_i$.
We also denote by $\alpha_i$ the index of $p_i$,
i.e., the two connected components of the partial normalization of the fibre $\Gamma^{\#}$ through $p_i$
intersect those $2g+2$ sections in $\alpha_i$ and $2g+2-\alpha_i\geq \alpha_i$ points respectively.

Let $\wt Y \to Y^{\#}$ be the resolution of singularities on $Y^{\#}$,
and $\tilde \pi:\,\wt S \to \wt Y$ the smooth double cover with branched divisor $\wt R$
as in Figure \ref{hyperellipticdiagram}.
The pullbacks of those $2g+2$ disjoint sections $\sigma_i$'s
are still disjoint sections of $\tilde h:\,\wt Y \to B$,
and by abuse of notation, we still denote them by $\sigma_i$'s.
Let $\wt R_h=\sum\limits_{i=1}^{2g+2} \sigma_i$.
Then $\wt R$ is the union of $\wt R_h$ and some disjoint $(-2)$-curves contained
in fibres of $\tilde h:\, \wt Y \to B$.

Let $\wt \Lambda=\{q_l\}$ be the set of points of $\wt Y$ which are nodes of fibres of $\tilde h$.
For a node $q_l \in \wt\Gamma$ in a fibre $\wt \Gamma$ of $\tilde h$, we also
define the index of $q_l$ to be $\beta_l$
if the two connected components of the partial normalization of the fibre $\wt\Gamma$ at $q_l$
intersect those $2g+2$ sections in $\beta_l$ and $2g+2-\beta_l\geq \beta_l$ points respectively.
Then a node $p_i$ in a fibre of $h^{\#}$ of index $\alpha_i$ with multiplicity $m_i$
would introduce $m_i$ nodes in the corresponding fibre of $\tilde h$ with the same indices $\alpha_i$.

Let $\tilde h':\, \wt Y \to D\cong \mathbb P^1$ be the morphism given in Proposition \ref{fibredprop}.
Let $\tilde\rho:\,\wt Y \to \hat Y$ be the largest contraction of `vertical' $(-1)$-curves
such that we still have a morphism $\hat h':\, \hat Y \to D$,
here `vertical' means such a curve is mapped to a point on $B$.
\begin{center}\mbox{}
\xymatrix{
 \wt Y \ar@{->}[rr]^-{\tilde\rho}\ar@{->}[dr]_-{\tilde h'} && \hat Y \ar@{->}[dl]^-{\hat h'}\\
 &D\cong \mathbb P^1&
}\end{center}
This means that any `vertical' $(-1)$-curve in $\hat Y$ is mapped surjectively onto $D$ by $\hat h'$.
Since $\wt R$ is contained in fibres of $\tilde h'$ by Proposition \ref{fibredprop},
$\hat R=\tilde\rho(\wt R)$ is contained in fibres of $\hat h'$.
So in particular any `vertical' $(-1)$-curve in $\hat Y$
is not contained in $\hat R\subseteq \hat Y$.
Let $\hat R_h=\tilde\rho(\wt R_h)$.

\begin{claim}\label{CRgeq2q_f+2}
For any `vertical' $(-1)$-curve $C$ in $\hat Y$,
$C\cdot\hat R \geq 2q_f+2$.
In particular,
$$C\cdot\hat R_h \geq 2q_f+1.$$
\end{claim}
\begin{proof}[Proof of the claim]
Note that $\hat R$ is the union of $\hat R_h$ and some curves in fibres of $\hat h$,
where $\hat h:\, \hat Y \to B$ is the induced morphism from $\tilde h:\,\wt Y \to B$.
Hence for any `vertical' $(-1)$-curve $C$,
let $\hat \Gamma$ be the fibre of $\hat h$ containing $C$.
Then $$C\cdot(\hat R-\hat R_h)\leq C\cdot \hat \Gamma = -C^2=1.$$
Therefore it suffices to prove $C\cdot\hat R \geq 2q_f+2$.

Let $C'\subseteq \wt S$ and $\wt C\subseteq \wt Y$ be the strict inverse image of $C$
on $\wt S$ and $\wt Y$ respectively.
Then by construction, $C'$ is mapped surjectively onto $B'$ by $\tilde f'$, and
$$C\cdot\hat R\geq \wt C \cdot \wt R.$$
Applying Hurwitz formula to the double cover $C'\to \wt C\cong \bbp^1$,
whose branched locus is at most $\wt C \cap \wt R$, one gets
\begin{equation*}
2g(C')-2\leq -4+\#(\wt C \cap \wt R).
\end{equation*}
As $C'$ is mapped surjectively onto $B'$,
$g(C')\geq g(B')=q_f$.
Hence
$$C\cdot\hat R\geq \wt C \cdot \wt R\geq \#(\wt C \cap \wt R)
\geq 2g(C')+2\geq 2q_f+2.$$
\end{proof}
Now we contract $\hat\rho:\,\hat Y \to Y$ to be a $\bbp^1$-bundle $h:\,Y\to B$ in such a way that
the order of any singularity of $R_h=\hat\rho(\hat R_h)$ is at most $g+1$.
It is easy to see that such a contraction exists.
\begin{center}\mbox{}
\xymatrix{
 \wt Y \ar@{->}[rr]^-{\tilde\rho}\ar@{->}[drr]_-{\tilde h} && \hat Y\ar[rr]^-{\hat\rho} \ar@{->}[d]^-{\hat h}
 &&Y\ar[dll]^-{h}\\
 &&B&&
}\end{center}
Let $\rho=\hat\rho\circ\tilde\rho$.
Then $\rho$ can be viewed as a sequence of blowing-ups
$\rho_l:\,Y_l \to Y_{l-1}$ centered at $y_{l-1}\in Y_{l-1}$ with $Y_{t+s}=\wt Y$,
$Y_s=\hat Y$, and $Y_0=Y$.
Let $R_{h,l}\subseteq Y_l$ be the image of $\wt R_h$,
$y_{l-1}$ a singularity of $R_{h,l-1}$ of order $n_{l-1}$.
Then one sees that each blowing-up $\rho_l$
creates a node $q\in \wt\Lambda$ with index
$\beta=n_{l-1}$. Hence
\begin{equation}\label{wtR=R}
\wt R_h^2=R_h^2-\sum_{q_l\in \wt\Lambda} \beta_l^2\,.
\end{equation}

By Claim \ref{CRgeq2q_f+2}, for $1\leq l\leq s$,
any blowing-up $\rho_l:\,Y_l\to Y_{l-1}$ is centered at a point $y_{l-1}$ with $n_{l-1}\geq 2q_f+1$.
In other words, for $1\leq l\leq s$,
each $\rho_l$ creates a node $q\in \wt\Lambda$ with index at least $2q_f+1$.
We divide the nodes $\wt \Lambda$ of fibres of $\tilde h$ into two parts:
one, denoted by $\wt \Lambda_{\hat\rho}$, is created by blowing-ups contained in $\hat \rho$;
the other one, denoted by $\wt \Lambda_{\tilde\rho}$, is created by blowing-ups contained in $\tilde \rho$.
Then
\begin{eqnarray}
\beta_l&\geq&2q_f+1,\qquad\quad\forall~q_l\in \wt \Lambda_{\hat\rho};\label{q_lgeq2q_f+1}\\
\hat R_h^2&=&R_h^2-\sum_{q_l\in \wt \Lambda_{\hat\rho}} \beta_l^2\,.\label{hatR=R}
\end{eqnarray}
Note that $\wt R_h$ consists of $2g+2$ disjoint sections $\sigma_i$'s.
According to \cite[Lemma 4.8]{ch88}, it follows that
\begin{equation}\label{wtR=bych88}
\wt R_h^2=\sum_{i=1}^{2g+2}\sigma_i\cdot \sigma_i
=-\sum_{p_i\in \Lambda}\frac{m_i\alpha_i(2g+2-\alpha_i)}{2g+1}
=-\sum_{q_l\in \wt \Lambda}\frac{\beta_l(2g+2-\beta_l)}{2g+1}.
\end{equation}
Combining \eqref{wtR=R}, \eqref{hatR=R} and \eqref{wtR=bych88}, one gets
\begin{equation}\label{hatR=wtR}
\hat R_h^2=\wt R_h^2+\sum_{q_l\in \wt \Lambda_{\tilde\rho}} \beta_l^2=
\sum_{q_l\in \wt \Lambda_{\tilde\rho}} \frac{(2g+2)\beta_l(\beta_l-1)}{2g+1}
-\sum_{q_l\in \wt \Lambda_{\hat\rho}}\frac{\beta_l(2g+2-\beta_l)}{2g+1}.
\end{equation}
Now according to Proposition \ref{fibredprop},
$\wt R_h\subseteq \wt R$ is contained in the fibres of $\tilde h'$.
By our construction, $\hat R_h=\tilde\rho(\wt R_h)$ is contained in the fibres of $\hat h'$.
In particular, $\hat R_h^2 \leq 0$. Hence by \eqref{hatR=wtR}, we obtain
\begin{equation}\label{xi11}
\sum_{q_l\in \wt \Lambda_{\tilde\rho}} \beta_l(\beta_l-1)
\leq \sum_{q_l\in \wt \Lambda_{\hat\rho}}\frac{\beta_l(2g+2-\beta_l)}{2g+2}.
\end{equation}

Let $\epsilon_k$ (resp. $\nu_k$) be the number of points $q_l\in \wt \Lambda$
of index $2k+1$ (resp. $2k+2$).
Then it is clear that
$\epsilon_k$ (resp. $\nu_k$) is also the number of points $p_i\in \Lambda$ of index $2k+1$ (resp. $2k+2$),
counted according to their multiplicity.
Hence (cf. \cite[(4.10)]{ch88})
\begin{equation}\label{invrelationdouble}
\xi_0(\Upsilon)=2\nu_0;\quad
\delta_i(\Upsilon)=\epsilon_i/2,~\forall\,1\leq i\leq [g/2];\quad
\xi_j(\Upsilon)=\nu_j,~\forall\,1\leq j\leq [(g-1)/2].
\end{equation}
Combining all together, one gets
\begin{eqnarray*}
&&\sum_{i= q_f}^{[g/2]}\frac{(2i+1)(2g+1-2i)}{g+1}\delta_i(\Upsilon)
+\sum_{j= q_f}^{[(g-1)/2]}\frac{2(j+1)(g-j)}{g+1}\xi_j(\Upsilon)\\
&=&\sum_{i= q_f}^{[g/2]}\frac{(2i+1)\big((2g+2)-(2i+1)\big)}{2g+2}\epsilon_i
+\sum_{j= q_f}^{[(g-1)/2]}\frac{(2j+2)\big((2g+2)-(2j+2)\big)}{2g+2}\nu_j\\
&=&\sum\frac{\beta_l(2g+2-\beta_l)}{2g+2}, \qquad
\text{the sum is taken over all $q_l\in \wt \Lambda$ with index $\beta_l\geq 2q_f+1$,}\\
&\geq& \sum_{q_l\in \wt \Lambda_{\hat\rho}}\frac{\beta_l(2g+2-\beta_l)}{2g+2},\quad~\,
\text{since any point $q_l\in \wt \Lambda_{\hat\rho}$ is of index $\beta_l\geq 2q_f+1$ by \eqref{q_lgeq2q_f+1},}\\
&\geq& \sum_{q_l\in \wt \Lambda_{\tilde\rho}} \beta_l(\beta_l-1), \qquad\qquad \text{by~}\eqref{xi11},\\
&\geq&\sum \beta_l(\beta_l-1),\qquad\qquad~~
\begin{aligned}
&\text{the sum is taken over all $q_l\in \wt \Lambda$ with index $\beta_l< 2q_f+1$,}\\
&\text{and such points are all contained in $\wt \Lambda_{\tilde\rho}$ by \eqref{q_lgeq2q_f+1},}
\end{aligned}
\\[0.1cm]
&=&2\nu_0+\sum_{i= 1}^{q_f-1}2i(2i+1)\epsilon_i
+\sum_{j= 1}^{q_f-1}2(j+1)(2j+1)\nu_j\\
&=& \xi_0(\Upsilon)+\sum_{i= 1}^{q_f-1}4i(2i+1)\delta_i(\Upsilon)
+\sum_{j= 1}^{q_f-1}2(j+1)(2j+1)\xi_j(\Upsilon).
\end{eqnarray*}
This completes the proof.
\end{proof}

The next lemma will be used in Section \ref{sectionconclusion}.
\begin{lemma}\label{d=2bars_i=0}
Let $f:\,S\to B$ be a non-isotrivial semi-stable family of
hyperelliptic curves of genus $g$ with singular locus
$\Upsilon/\Delta$,
and $d$ be defined in \eqref{degf_1'=GammaGamma}.
If $d=2$, then
$$\delta_i(\Upsilon)=\xi_i(\Upsilon)=0,\quad \forall~1\leq i\leq q_f-1.$$
\end{lemma}
\begin{proof}
As what we did in the proof of Proposition \ref{boundofxi_0prop},
we may assume that the stable model $f^{\#}:\,S^{\#} \to B$
comes from an admissible cover.
We also use the same symbols and notations introduced there.

According to \eqref{invrelationdouble}, it suffices to prove that
$$\epsilon_i=\mu_i=0,\quad \forall~1\leq i\leq q_f-1.$$
Since any point $q_l\in \wt \Lambda_{\hat\rho}$ is of index $\beta_l\geq 2q_f+1$,
it is enough to prove that for any point $q_l\in \wt \Lambda_{\tilde\rho}$ is of index $\beta_l=2$.

Let $k=\beta_l$. Assume that $q_l\in \wt \Lambda_{\tilde\rho}$ is created by a blowing-up
$\rho_l:\,Y_l \to Y_{l-1}$ centered at $y_{l-1}\in R_{h,l-1}\subseteq Y_{l-1}$.
Then $y_{l-1}$ is a singularity of $R_{h,l-1}$ of order $k$.
Let $\{\tau_1,\cdots,\tau_{k}\}\subseteq R_{h,l-1}$ be those sections passing through $y_{l-1}$,
and $h_{l-1}:\,Y_{l-1} \to B$, $h_{l-1}':\,Y_{l-1} \to D$ the induced morphisms.
Since $R_{h,l-1}$ is contained in fibres of $h_{l-1}'$ by construction,
and $\{\tau_1,\cdots,\tau_{k}\}$ have a common point $y_{l-1}$,
it follows that $\{\tau_1,\cdots,\tau_{k}\}$ must be contained in one fibre of $h_{l-1}'$.
\begin{center}\mbox{}
\xymatrix{
  \wt S \ar[d]\ar[rr]^{\tilde f'} \ar@/_7mm/"3,1"_{\tilde f} && B'\ar[d]\\
  Y_{l-1} \ar[d]^{h_{l-1}}\ar[rr]^{\tilde h_{l-1}'} &&D\\
  B&&
}\end{center}

Denote by $\hat \Gamma$ the fibre of $h_{l-1}'$ containing $\{\tau_1,\cdots,\tau_{k}\}$,
and by $\wt F'$ the corresponding fibre of $\tilde f':\,\wt S \to B'$.
Then it is not difficult to see that
$$2=d=\deg \left(\tilde f|_{\wt F'}\right)=\deg \left(h_{l-1}|_{\hat \Gamma}\right)\geq
\sum_{i=1}^k\deg \left(h_{l-1}|_{\tau_i}\right)=k,$$
where $\tilde f|_{\wt F'}:\,\wt F' \to B$
(resp. $h_{l-1}|_{\hat \Gamma}:\,\hat \Gamma \to B$,
resp. $h_{l-1}|_{\tau_i}:\,\tau_i \to B$) is the restricted map of $\wt F'$
(resp. $\hat \Gamma$, resp $\tau_i$) to $B$.
This completes the proof.
\end{proof}

\subsection{Proof of Theorem \ref{thmlower} for $q_f>0$}\label{subsectpfthmlower}
The subsection aims to prove Theorem \ref{thmlower} for the case $q_f>0$.
It is based on the formula \eqref{formulaofdegomega}
given by Cornalba-Harris
and \eqref{boundofxi_0eqn} obtained in Proposition \ref{boundofxi_0prop}.

We consider first the case $\Delta_{nc}\neq \emptyset$.
By \eqref{formulaofdegomega} and \eqref{formulaofomega^2}, one gets
$$\begin{aligned}
\omega_{S/B}^2-\frac{4(g-1)}{g-q_f}\deg \left(f_*\Omega^1_{S/B}(\Upsilon)\right)=&
-\frac{(g-1)q_f}{(2g+1)(g-q_f)}\xi_0(\Upsilon)\\
&+\sum_{i=1}^{[g/2]}\left(\frac{4(2g-3q_f+1)i(g-i)}{(2g+1)(g-q_f)}-1\right)\delta_i(\Upsilon)\\
&+\sum_{j=1}^{[(g-1)/2]}\left(\frac{2(2g-3q_f+1)(j+1)(g-j)}{(2g+1)(g-q_f)}-2\right)\xi_j(\Upsilon).
\end{aligned}$$
Combining this with \eqref{boundofxi_0eqn}, one gets
\begin{eqnarray*}
&&\omega_{S/B}^2-\frac{4(g-1)}{g-q_f}\deg \left(f_*\Omega^1_{S/B}(\Upsilon)\right)\\
&\geq& \sum_{i= 1}^{q_f-1} a_i \delta_i(\Upsilon)+
\sum_{i= q_f}^{[g/2]} b_i \delta_i(\Upsilon)+
\sum_{j= 1}^{q_f-1} c_j \xi_j(\Upsilon)+
\sum_{j= q_f}^{[(g-1)/2]} d_j \xi_j(\Upsilon),
\end{eqnarray*}
where
$$\left\{
\begin{aligned}
&a_i=\left(\frac{4(2g-3q_f+1)i(g-i)}{(2g+1)(g-q_f)}-1\right)
     +\frac{(g-1)q_f}{(2g+1)(g-q_f)} \cdot 4i(2i+1),\\
&b_i=\left(\frac{4(2g-3q_f+1)i(g-i)}{(2g+1)(g-q_f)}-1\right)
     -\frac{(g-1)q_f}{(2g+1)(g-q_f)} \cdot \frac{(2i+1)(2g+1-2i)}{g+1},\\
&c_j=\left(\frac{2(2g-3q_f+1)(j+1)(g-j)}{(2g+1)(g-q_f)}-2\right)
     +\frac{(g-1)q_f}{(2g+1)(g-q_f)} \cdot 2(j+1)(2j+1),\\
&d_j=\left(\frac{2(2g-3q_f+1)(j+1)(g-j)}{(2g+1)(g-q_f)}-2\right)
     -\frac{(g-1)q_f}{(2g+1)(g-q_f)} \cdot \frac{2(j+1)(g-j)}{g+1}.
\end{aligned}\right.
$$
If $q_f=1$, then
$$
\begin{aligned}
b_1&=\frac{3g-6}{g+1};&\quad&\\
b_i&=\frac{4i(g-i)-g-2}{g+1}\geq \frac{7g-18}{g+1}, && \forall~2\leq i\leq [g/2];\\
d_j&=\frac{2\big((j+1)(g-j)-(g+1)\big)}{g+1}\geq 0,&& \forall~1\leq j\leq [(g-1)/2].
\end{aligned}
$$
If $q_f\geq 2$, then
$$
\begin{aligned}
a_1&\geq \frac{3g^2-(8q_f+1)g+10q_f-4}{(g+1)(g-q_f)};&\quad&\\
a_i&\geq \frac{7g^2-(16q_f+9)g+34q_f-16}{(g+1)(g-q_f)},&\quad&\forall~2\leq i\leq q_f-1;\\
b_i&\geq \frac{7g^2-(16q_f+9)g+34q_f-16}{(g+1)(g-q_f)}, && \forall~q_f\leq i\leq [g/2];\\
c_j&\geq 0,&& \forall~1\leq j\leq q_f-1;\\
d_j&\geq 0,&& \forall~q_f\leq j\leq [(g-1)/2].
\end{aligned}
$$
Hence \eqref{eqnlower} holds for $\Delta_{nc}\neq\emptyset$.
\vspace{0.2cm}

Now we consider the case that $\Delta_{nc}=\emptyset$.
Note that in this case,
\begin{equation}\label{xi=0}
\xi_j(\Upsilon)=0, \qquad\qquad \forall~ 0 \leq j \leq [(g-1)/2].
\end{equation}
Hence by \eqref{formulaofdegomega} and \eqref{formulaofomega^2}, we get
$$\omega_{S/B}^2-\frac{4(g-1)}{g-q_f}\deg \left(f_*\Omega^1_{S/B}(\Upsilon)\right)=
\sum_{i=1}^{[g/2]} \left(\frac{4(2g+1-3q_f)i(g-i)}{(2g+1)(g-q_f)}-1\right) \delta_i(\Upsilon).$$
Hence \eqref{eqnlower} holds too for $\Delta_{nc}=\emptyset$.
If moreover $q_f\geq 2$, then according to \eqref{boundofxi_0eqn} and \eqref{xi=0},
we get
$$\sum_{i=q_f}^{[g/2]} \frac{(2i+1)(2g+1-2i)}{g+1} \cdot \delta_i(\Upsilon)\geq
\sum_{i=1}^{q_f-1} 4i(2i+1) \cdot \delta_i(\Upsilon).$$
So \eqref{eqnlowerSigma_0=0} is proved.
\qed

\section{Flat part of $R^1f_*\mathbb C$ for hyperelliptic families}\label{sectionflathyper}
The purpose of the section is to prove Theorem \ref{thmFtrivial}.
It is based on two lemmas.
The first one is  Lemma \ref{lemmapeters},
coming from a discussion with Chris Peters, on the global invariant
cycle with unitary locally constant coefficient,
which generalizes Deligne's original theorem with the constant coefficient.
The second one is Bogomolov's lemma on Kodaira dimension of an invertible subsheaf
of the sheaf of logarithmic differential forms on a smooth projective surface
(cf. \cite[Lemma 7.5]{sakai80}).

Let $f:\,S \to B$ be a non-isotrivial
semi-stable family of hyperelliptic curves of genus $g$,
$\Upsilon \to \Delta$ the singular fibres of $f$,
and $S^0 \to B\setminus \Delta$ the smooth part of $f$.
The direct image sheaf $R^1f_*\mathbb C_{S^0}$ is a local system on $B\setminus \Delta$,
which underlies a variation of Hodge structures of weight one.
Let $$\left(E\cong f_*\Omega^1_{S/B}(\log\Upsilon)\oplus R^1f_*\mathcal O_S,~\theta\right)$$
be the Higgs bundle by taking the graded bundle of the Deligne extension of
$R^1f_*\mathbb C_{S^0}\otimes\mathcal O_{B\setminus \Delta}$.
According to \cite{fujita78} or \cite{kollar87},
$$\big(E,~\theta\big)=\big(A,~\theta|_A\big)\oplus \big(F,~0\big)$$
where $A^{1,0}$ is an ample vector bundle over $B$ and  $F^{1,0}$ is a flat vector bundle
coming from a representation of the fundamental group
$\tilde\rho_{F}:~\pi_1\big(B\setminus\Delta\big)\lra U(r)$
into a unitary group of rank $r=\rank F^{1,0}$.
Note that the monodromy around $\Delta$ is unipotent, since $f$ is semi-stable.
Hence $\tilde\rho_F$ actually factors through $\pi_1(B)$:
$$
\xymatrix{
\pi_1\big(B\setminus\Delta\big) \ar[rr]^-{\tilde\rho_F}\ar[dr]_-{i_*}&& U(r)\\
&\pi_1(B)\ar[ur]_-{\rho_F}&
}$$

\begin{theorem}\label{theoremq_f=rankF10}
Let $f:S\to B$ be a semi-stable family of hyperelliptic curves over $B$ as above.
Then after a suitable base change which is unbranched over $B\setminus \Delta$,
$F^{1,0}$ is a trivial bundle, i.e.,
$$F^{1,0}=\bigoplus\limits_{i=1}^{r} \mathcal O_B,\qquad\qquad {\rm where~} r=\rank F^{1,0}.$$
\end{theorem}

Note that $F^{1,0}$ and $F^{0,1}$ are dual to each other.
Hence Theorem \ref{thmFtrivial} follows from Theorem \ref{theoremq_f=rankF10}.
Indeed, from Theorem \ref{theoremq_f=rankF10} it follows that the image of $\tilde\rho_F$ is finite.
Because $\tilde\rho_F$ factors through $\pi_1(B)$
and $i_*$ is surjective, one gets that $\rho_F$ has also finite image.
It implies that after a suitable base change which is unbranched over $B$,
$F^{1,0}$ (hence also its dual $F^{0,1}$) becomes a trivial bundle.
So it remains to prove Theorem \ref{theoremq_f=rankF10}.

\vspace{0.15cm}
The proof of Theorem \ref{theoremq_f=rankF10} depends on
the following general statement on
the global invariant cycle with unitary locally constant coefficient,
The proof stated below comes from a discussion with Chris Peters. We thank him very much.
\begin{lemma}\label{lemmapeters}
Let $f: X_0\to B_0$ be a smooth proper morphism.
Take $ X\supset X_0$ to be a smooth compactification of $X_0$,
and let $\mathbb U$ be a locally constant sheaf $\mathbb U$ over $ X$,
which comes from a representation of $\pi_1(X, *)$ into the unitary group $U(n)$.
Then the canonical morphism:
  $$ H^k( X,\mathbb U)\lra H^0(B_0, R^kf_*\mathbb U)$$
  is surjective.
\end{lemma}
\begin{proof}
The proof is, in fact, along the same line as what Deligne did for the original case
when $\mathbb U=\mathbb Q$ (cf. \cite[\S\,4.1]{deligne71}).

The unitary locally constant sheaf $\mathbb U$ on $X$ carries in a natural way a polarized variation
of Hodge structure, say, of pure type $(0,0)$. Hence
it follows from Saito's theory that there is an induced pure  Hodge structure of weight $k$ on
$H^k(X, \mathbb U)$ as well
as on $H^k(X_b,\mathbb U|_{X_b})$ where $X_b$ is any (smooth projective) fibre of $f:X_0\to B_0$.

We first show  the "edge-homomorphism"
$$H^k(X_0, \mathbb U)\lra H^0(B_0,R^kf_*\mathbb U)$$ is surjective
by the following argument from the proof of \cite[Proposition 1.38]{peterssteenbrink}.

It is not difficult to see that we only need to find a class $h\in H^2(X_0, \mathbb{Q})$
such that cup-products satisfy the hard Lefschetz property,
i.e., the following homomorphism is an isomorphism
for any $0\leq k\leq m$, where $m$ is the dimension of a general fibre of $f$
$$
[\cup \,h]^k: ~R^{m-k}f_*\mathbb{U}\longrightarrow R^{m+k}f_*\mathbb{U}.
$$
As for the class $h\in H^2(X_0, \mathbb{Q})$, we just take an embedding $X\hookrightarrow \mathbb{P}^N$,
and let $h$ be the restriction of the hyperplane class.
Then the hard Lefschetz property can be verified fiber-by-fiber.
On each fiber the natural locally constant metric on $\mathbb{U}$
induces a Hodge decomposition of the cohomology with coefficients in $\mathbb{U}$,
hence the hard Lefschetz property holds.
So, we show that the above "edge-homomorphism" is surjective.

Hence, it also induces surjective morphisms
between the weight-filtrations  of the both cohomologies, in particular, on the lowest weight $k$
$$ W_k(H^k(X_0,\mathbb U))\twoheadrightarrow W_k(H^0(B_0, R^kf_*\mathbb U)).$$

Since the restriction morphism (as monodromy invariant)
$$ H^0(B_0,R^kf_*\mathbb U)\to H^k(X_b,\mathbb U|X_b)$$
is injective and $H^k(X_b,\mathbb U|X_b)$ carries a pure Hodge structure of weigh-$k$,
$H^0(B_0, R^kf_*\mathbb U)$ carries a pure Hodge structure of weight-$k.$
So the above surjective morphism becomes
$$ W_k(H^k(X_0,\mathbb U))\twoheadrightarrow W_k(H^0(B_0, R^kf_*\mathbb U))=H^0(B_0, R^kf_*\mathbb U).$$

Finally, according to \cite{peterssaito12},
$W_k (H^k(X_0,\mathbb U))$ is nothing but the image of the restriction homomorphism
$$H^k( X,\mathbb U)\to H^k(X_0,\mathbb U).$$

Put the above two surjective morphisms together, we obtain a surjective morphism:
$$H^k(X,\mathbb U)\twoheadrightarrow W_k(H^k(X_0,\mathbb U))\twoheadrightarrow H^0(B_0, R^kf_*\mathbb U).$$
The proof is finished.
\end{proof}

\begin{corollary}\label{corollaylifting}
Let $f:S\to B$ be a semi-stable family of projective curves
(not necessarily hyperelliptic) over a smooth projective curve $B,$
with semi-stable singular fibres $f:\Upsilon\to \Delta.$ Let
$ S^0=S\setminus\Upsilon.$ Given a  vector subbundle
$\mathcal U\subseteq f_*\Omega^1_{S/B}(\log\Upsilon),$ which underlies a unitary locally constant subsheaf
$\mathbb U\subseteq \mathbb V_{\mathbb C} \triangleq R^1f_*\mathbb C_{S^0},$
then  it lifts to a morphism
$$ f^*\mathcal U\to \Omega^1_S,$$
such that the induced canonical morphism
$$\mathcal U\to f_*\Omega^1_S\to f_*\Omega^1_S(\log\Upsilon)\to f_*\Omega^1_{S/B}(\log\Upsilon)$$
coincides with the subbundle $\mathcal U\subseteq  f_*\Omega^1_{S/B}(\log\Upsilon).$
\end{corollary}
\begin{proof}
Since the local monodromy of $\mathbb V$ around $\Delta$ is unipotent and
the local monodromy of the subsheaf $\mathbb U$ around $\Delta$ is semisimple, so $\mathbb U$
extends on $B$ as a locally constant sheaf.  The morphism
$ \mathbb U\subset \mathbb V_{\mathbb C}$ corresponds to
a  section
$$\eta\in H^0(B\setminus\Delta,\mathbb V_{\mathbb C}\otimes\mathbb U^\vee)
=H^0(B\setminus\Delta, R^1f_*(\mathbb C_{S^0}\otimes f^*\mathbb U^\vee)).$$
Applying Lemma \ref{lemmapeters}, $\eta$ lifts to a class $\tilde \eta\in H^1(S, f^*\mathbb U^\vee)$
under the canonical morphism
$$ H^1(S, f^*\mathbb U^\vee)\to H^0(B\setminus \Delta, R^1f_*(\mathbb C_{S^0}\otimes f^*\mathbb U^\vee)).$$
Note that this canonical morphism is a morphism between  pure Hodge structures of weight-1,
and by the construction
$\eta$ is of type (1,0), so $\tilde \eta$ is of type (1,0), i.e.,
$$\tilde \eta\in H^0(S,\Omega^1_S\otimes f^*\mathcal U^\vee),$$
which corresponds to a morphism
$$ f^*\mathcal U\to \Omega^1_S,$$
such that under the canonical morphism it goes back to $\mathcal U\subset f_*\Omega^1_{S/B}(\log\Upsilon).$
\end{proof}

In the rest part of this section, we prove Theorem \ref{theoremq_f=rankF10}.
It follows from Corollary \ref{corollaylifting} and
Bogomolov's lemma on Kodaira dimension of an invertible subsheaf
in the sheaf of logarithmic differential forms on a smooth projective surface
(cf. \cite[Lemma\,7.5]{sakai80}).

According to Corollary \ref{corollaylifting},
for any flat vector subbundle $\mathcal U\subseteq F^{1,0}$,
there is a sheaf morphism $f^*\mathcal U \to \Omega^1_S$,
such that the induced canonical morphism
$$\mathcal U\to f_*\Omega^1_S\to f_*\Omega^1_S(\log\Upsilon)\to f_*\Omega^1_{S/B}(\log\Upsilon)$$
coincides with the subbundle $\mathcal U\subseteq  f_*\Omega^1_{S/B}(\log\Upsilon).$

Let $\tilde\pi:\,\wt S \to \wt Y$ be the smooth double cover described in Figure \ref{hyperellipticdiagram},
and $\vartheta:\,\wt S \to S$ be the blowing-ups.
By pulling back, we obtain a sheaf morphism
$$\tilde f^*\mathcal U=\vartheta^*f^*\mathcal U \lra \Omega^1_{\wt S},
\qquad\qquad\text{where~}\tilde f=f\circ \vartheta,$$
which corresponds to an element
$$\tilde\eta\in H^0(\wt S, \Omega^1_{\wt S}\otimes \tilde f^*\mathcal U^{\vee}).$$
By pushing-out, we also obtain an element (where $\tilde h:\,\wt Y \to B$ is the induced morphism)
$$\tilde\pi_*(\tilde\eta)\in H^0\Big(\wt Y, \tilde\pi_*\left(\Omega^1_{\wt S}\otimes \tilde f^*\mathcal U^{\vee}\right)\Big)=H^0\Big(\wt Y, \tilde\pi_*\left(\Omega^1_{\wt S}\right)\otimes \tilde h^*\mathcal U^{\vee}\Big).$$
So one gets a sheaf morphism
$$\tilde h^*\mathcal U \to \tilde\pi_*\left(\Omega^1_{\wt S}\right).$$
The Galois group ${\rm Gal}(\wt S/\wt Y)\cong \mathbb Z_2$ acts on
$$\tilde\pi_*\left(\Omega^1_{\wt S}\right).$$
One obtains the eigenspace decomposition
$$ \tilde h^*(\mathcal U)\to \tilde\pi_*\left(\Omega^1_{\wt S}\right)_1,
\qquad \tilde h^*(\mathcal U)\to \tilde\pi_*\left(\Omega^1_{\wt S}\right)_{-1}.$$

\begin{lemma}\label{LR=0}
The image of the map
$$\varrho:~\tilde h^*(\mathcal U)\to \tilde\pi_*\left(\Omega^1_{\wt S}\right)_{-1}.$$
is an invertible subsheaf $M$
such that $M$ is numerically effective (nef) and $M^2=0$.
Let $\wt\Gamma$ be a general fibre of $\tilde h$,
and $D$ be any component of the branch divisor $\wt R\subseteq \wt Y$
of the double cover $\tilde\pi: \wt S\to \wt Y$. Then
$$ M\cdot D=0,\qquad\rank \mathcal U=\dim H^0(\wt\Gamma, \mathcal O_{\wt\Gamma}(M)).$$
\end{lemma}
\begin{proof}
First of all, we want to show that $\varrho\neq 0$. It is known that
\begin{equation}\label{pi_*fenjie}
\tilde\pi_*\left(\Omega^1_{\wt S}\right)_1=\Omega^1_{\wt Y},\qquad
\tilde\pi_*\left(\Omega^1_{\wt S}\right)_{-1}=\Omega^1_{\wt Y}\left(\log (\wt R)\right)(-\wt L),
\end{equation}
where $\wt R\equiv2\wt L$ ($\equiv$ stands for linearly equivalent)
is the defining data of the double cover $\tilde\pi:\, \wt S \to \wt Y$.
Note that the induced map
\begin{eqnarray}
\mathcal U=\tilde h_*\tilde h^*\mathcal U
&\lra&~~ \tilde h_*\bigg(\tilde\pi_*\left(\Omega^1_{\wt S}\right)_1\oplus \tilde\pi_*\left(\Omega^1_{\wt S}\right)_{-1}\bigg)\nonumber\\
&&=\tilde h_*\bigg(\Omega^1_{\wt Y}\left(\log (\wt R)\right)(-\wt L)\bigg)\subseteq E^{1,0}\label{Ninclusion}
\end{eqnarray}
is just the inclusion $\mathcal U\subseteq E^{1,0}$.
Hence in particular, $\varrho\neq0$.

We claim that the image of $\varrho$
is a subsheaf of rank one. Otherwise,
it is of rank two,
and so the second wedge product
$$\wedge^2 \tilde h^* \mathcal U\overset{\wedge^2\varrho}\lra \wedge^2\left( \tilde\pi_*(\Omega^1_{\wt S})_{-1}\right)=\omega_{\wt Y}$$
is a non-zero map. Note that the image of that map
is a quotient sheaf of  $\wedge^2g^*\mathcal U$ coming from a unitary local system, so the image sheaf
is semi-positive. But, it is impossible, since $\omega_{\wt Y}$ can not contain
any non-zero semi-positive subsheaf.

So the image of $\varrho$ is a rank one subsheaf $M\otimes I_{Z}$,
where $M$ is an invertible subsheaf and $\dim Z=0$.
Actually, $Z =\emptyset$;
otherwise by a suitable blowing-up $\rho:\,X\to \wt Y$, we may assume the image of
$\rho^*\tilde h^*\mathcal U$ is $\rho^*(M)\otimes (-E)$,
where $E$ is a combination of the exceptional curves.
As $\mathcal U$ comes from a unitary local system, we get $\rho^*(M)\otimes (-E)$ is semi-positive and
$$0\leq (\rho^*(M)-E)^2=M^2+E^2.$$
So $M$ is semi-positive and $M^2 \geq -E^2>0$, which implies that the Kodaira dimension of $M$ is $2$.
On the other hand, by \eqref{pi_*fenjie}, we get the following inclusion of sheaves,
\begin{equation}\label{Linclusion}
\mathcal O_{\wt Y}(\wt L)\otimes M\subseteq \Omega^1_{\wt Y}\left(\log (\wt R)\right).
\end{equation}
As $2\wt L\equiv \wt R$ is effective,
the Kodaira dimension of $\wt L\otimes M$ is also $2$, which
is impossible by Bogomolov's lemma (cf. \cite[Lemma 7.5]{sakai80}).

Hence the image of $\varrho$ is an invertible subsheaf $M$,
which is semi-positive since it is a quotient sheaf of a vector bundle coming from a unitary local system.
Note that we still have the inclusion \eqref{Linclusion}.
So again by Bogomolov's lemma (cf. \cite[Lemma 7.5]{sakai80}), we get
$$M^2=0,\text{\quad and \quad} M\cdot D=0.$$
Finally, according to \eqref{Ninclusion}, we have
$$\rank \mathcal U=\dim H^0\big(\wt\Gamma, \mathcal O_{\wt\Gamma}(M)\big).$$
\end{proof}

\begin{proof}[Proof of Theorem {\rm\ref{theoremq_f=rankF10}}]
Similarly as the proof of Proposition \ref{boundofxi_0prop},
we may restrict ourselves to the situation that
the induced double cover $\pi^{\#}:\,S^{\#}\to Y^{\#}$ in Figure \ref{hyperellipticdiagram}
comes from an admissible double cover (cf. \cite{ch88} or \cite{harrismumford82}).
In fact, for any semi-stable family $f:\, S \to B$ of hyperelliptic curves,
we may get a family of
admissible covers by a base change which is unramified over $B\setminus \Delta$
and blowing-ups of singular points in the fibres.

We first prove that in such a situation,
$F^{1,0}$ is a direct sum of line bundles $\mathcal F_i$ on $B$, i.e.,
\begin{equation}\label{decompintolineb}
F^{1,0}=\bigoplus\limits_{i=1}^{r} \mathcal F_i,\qquad\qquad {\rm where~} r=\rank F^{1,0}.
\end{equation}
By assumption, the branched divisor $\wt R \subseteq \wt Y$ of the induced smooth double cover
$\tilde\pi:\, \wt S \to \wt Y$
is a union of $2g+2$ sections and some curves contained in fibres of $\tilde h:\, \wt Y \to B$.
Let $D$ be such a section, and
$$F^{1,0}=\bigoplus\limits_{i=1}^{t}\,\mathcal U_i,$$
be a decomposition into irreducible components.

If $\rank \mathcal U_i \geq 2$ for some $i$,
then by setting $\mathcal U=\mathcal U_i$ in Lemma \ref{LR=0},
we obtain $M\cdot D=0$.
Equivalently if we write $\mathcal F=\mathcal O_D(M)$, then $$\deg \calf=0.$$
Note that $M$ is a quotient of $\tilde h^*\mathcal U$.
As $D$ is a section, $D\cong B$. Hence we may view $\calf$ is an invertible
subsheaf on $B$, which is a quotient of $\mathcal U$.
As $\mathcal U$ comes from a unitary local system, $\mathcal U$ is poly-stable.
Thus $\mathcal U=\calf \oplus \mathcal U'$,
which is a contradiction, since $\mathcal U=\mathcal U_i$ is irreducible by our assumption.
Hence we obtain the required decomposition \eqref{decompintolineb}.

Now applying \cite[\S\,4.2]{deligne71} or \cite[Theorem 3.4]{bar98},
we get that $\mathcal F_i$ is torsion in ${\rm Pic}^0(B)$.
Hence after a suitable base change which is unbranched over $B$,
$\calf_i \cong \mathcal O_B$.
Therefore, the proof is finished.
\end{proof}

\section{Conclusions}\label{sectionconclusion}
The purpose of the section is to prove our main results,
Theorems \ref{mainthm1} and \ref{mainthm}.
As illustrated in Section \ref{sectionintroduction},
the proof follows from the Arakelov equality
for the characterization for $f$ being a Kuga family
together with those bounds on $\omega_{S/B}^2$ given in
Theorems \ref{thmlower1}, \ref{thmupper} and \ref{thmlower}.
\vspace{0.1cm}

Let $f:\, S\to B$ be a Kuga family of curves of genus $g\geq 2$.
Let $\Upsilon/\Delta$ denote  semi-stable singular fibres,
$\Upsilon_{c}/\Delta_{c}$ denote those singular fibres with compact Jacobians,
and $\Upsilon_{nc}/\Delta_{nc}$
correspond to singular fibres with non-compact Jacobians.
Then the logarithmic
Higgs bundle associated to the VHS of $f$ is decomposed as Higgs subbundles
$$\left(f_*\Omega^1_{S/B}(\log \Upsilon)\oplus R^1f_*\mathcal O_S,~\theta\right)
=\left(A^{1,0}\oplus A^{0,1},~\theta|_{A^{1,0}}\right)\oplus \left(F^{1,0}\oplus F^{0,1},~0\right),$$
 where
$$\theta|_{A^{1,0}}: A^{1,0}\to A^{0,1}\otimes\Omega^1_B(\log\Delta_{nc})$$
is described on Page \pageref{defoftheta}.
As $f$ is a Kuga family, $\theta|_{A^{1,0}}$ is an isomorphism by \cite{mvz12} or \cite{vz04}.
In other words, one has
\begin{equation}\label{arakelovrankA'}
\deg \left(f_*\Omega^1_{S/B}(\Upsilon)\right)= \deg A^{1,0}
=\frac{\rank A^{1,0}}{2}\cdot\deg\left(\Omega^1_{B}(\log\Delta_{nc})\right).
\end{equation}

\subsection{Proof of Theorem \ref{mainthm1}}
The case $\Delta_{nc}\neq \emptyset$ is already proved in Section \ref{sectionintroduction}.
We consider here only the case $\Delta_{nc}=\emptyset$.
In this case, $\Delta=\Delta_c$ and $\Upsilon=\Upsilon_c$.

Since the logarithmic Higgs bundle associated to the family has strictly maximal Higgs field,
by \eqref{arakelovrankA'}, we have
$$\deg \left(f_*\Omega^1_{S/B}(\Upsilon)\right)
=\frac{g}{2}\cdot\deg\Omega^1_B.$$
Combining this with \eqref{eqnlower1}, one has
\begin{equation}\label{pfofmainthm2}
\omega_{S/B}^2\geq (2g-2)\cdot \deg\Omega^1_{B}
+\frac{3g-4}{g}\delta_1(\Upsilon)+\frac{7g-16}{g}\delta_h(\Upsilon).
\end{equation}
Together with \eqref{eqnupper}, we get
\begin{equation}\label{pfofmainthm1}
0\geq \frac{g-4}{g}\cdot\big(\delta_1(\Upsilon)+4\delta_h(\Upsilon)\big).
\end{equation}
Note that both $\delta_1(\Upsilon)$ and $\delta_h(\Upsilon)$ are non-negative.
If one of them is positive, then $g\leq 4$ by \eqref{pfofmainthm1}.
If $\delta_1(\Upsilon)=\delta_h(\Upsilon)=0$,
then $\Upsilon=\Upsilon_c=\emptyset$,
i.e.,  $\Delta=\Delta_c=\emptyset$.
Hence the inequality \eqref{eqnupper} in Theorem \ref{thmupper} is strict.
So \eqref{pfofmainthm1} is also strict, which is impossible. \qed
\begin{remarks}
(i).~ If $g=4$, then \eqref{eqnupper} must be an equality according the proof above.
Hence $S\setminus\left(\bigcup\limits_{p\in\Delta_c}D_{p}\right)$
is a ball quotient by Remark \ref{miyaoka12},
where $D_p$ is defined in \eqref{miyaoka14}.
We refer to Example \ref{exshimurag=4} for such an example.

(ii).~ There is also another way to show that $\delta_1(\Upsilon)$
and $\delta_h(\Upsilon)$ cannot be zero simultaneously
if $f$ is a Kuga family with $\Upsilon=\Upsilon_c$ and strictly maximal Higgs field.
Assume $\delta_1(\Upsilon)=\delta_h(\Upsilon)=0$,
then \eqref{pfofmainthm1} is an equality, which implies that \eqref{pfofmainthm2} is also an equality.
So we have
$$\omega_{S/B}^2= (2g-2)\cdot \deg\Omega^1_{B}
=\frac{4(g-1)}{g}\cdot\deg \left(f_*\Omega^1_{S/B}(\Upsilon)\right).$$
This implies that $f$ must be a hyperelliptic family by \cite[Theorem\,(4.12)]{ch88}.
However, for a hyperelliptic family with no singular fibres,
$\deg \left(f_*\Omega^1_{S/B}(\Upsilon)\right)=0$ by \eqref{formulaofdegomega},
which is impossible.
\end{remarks}

\subsection{Proof of Theorem \ref{mainthm}}
The subsection is aimed to prove Theorem \ref{mainthm}.
The idea is similar to that of proving Theorem \ref{mainthm1}.
It is based on the Arakelov equality
for the characterization for $f$ being a Kuga family
together with those bounds on $\omega_{S/B}^2$ given in
Theorems \ref{thmupper} and \ref{thmlower}.
We also need the fact that the rank of flat part of
the logarithmic Higgs bundle associated to $f$ is exactly the relative irregularity $q_f$
up to some unbranched base change.
\vspace{0.1cm}

We assume in this subsection that
$f:\, S\to B$ is a Kuga family of hyperelliptic curves.
According to Theorem \ref{thmFtrivial} together with Deligne's global invariant cycle theorem
(cf. \cite[\S\,4.1]{deligne71}) or Fujita's decomposition theorem (cf. \cite[Theorem\,3.1]{fujita78}),
after replacing $B$ by some suitable unbranched cover, one has
$$\rank A^{1,0}=g-q_f.$$
Note that the property that the Higgs field $\theta$ is maximal remains true
under any unbranched base change.
Hence the Arakelov equality
for the characterization for $f$ being a Kuga family reads as
\begin{equation}\label{arakelovinequalityg-q_f}
\deg \left(f_*\Omega^1_{S/B}(\Upsilon)\right)= \frac{g-q_f}{2}\cdot\deg\Omega^1_B(\log \Delta_{nc}).
\end{equation}
By the definition,
\begin{equation}\label{deltacleq}
0\leq \delta_i(\Upsilon_{c}) \leq \delta_i(\Upsilon), \qquad \forall~1\leq i \leq [g/2].
\end{equation}
Combining this with \eqref{arakelovinequalityg-q_f}, \eqref{eqnupper} and \eqref{eqnlower}, one obtains
that if $\Delta_{nc}\neq \emptyset$, then
\begin{equation}\label{sigma_ncbukong<=}
0>\frac{g^2-(6q_f+3)g+12q_f-4}{(g+1)(g-q_f)} \delta_1(\Upsilon)
+\frac{4g^2-(13q_f+12)g+37q_f-16}{(g+1)(g-q_f)} \delta_h(\Upsilon),
\end{equation}
and if $\Delta_{nc} =\emptyset$, then
\begin{equation}\label{sigma_nckong<=}
\hspace{-0.3cm}0\geq \left(\frac{4(2g+1-3q_f)(g-1)}{(2g+1)(g-q_f)}-3\right) \delta_1(\Upsilon)+
\sum_{i=2}^{[g/2]}
\left(\frac{4(2g+1-3q_f)i(g-i)}{(2g+1)(g-q_f)}-4\right) \delta_i(\Upsilon).
\end{equation}
One might imagine that it is impossible if $g$ is large enough,
since $\delta_i(\Upsilon)$'s are non-negative;
and one of $\delta_i(\Upsilon)$'s must be positive
by \eqref{formulaofdegomega} if $\Delta_{nc} =\emptyset$.
In other words, there should not exist a Kuga family of hyperelliptic curves of genus $g$
when $g$ is sufficiently large.
The detail computation is complicated and occupies the rest of the section.

\begin{proof}[Proof of Theorem {\rm \ref{mainthm}}]
We divide the proof into two cases: $\Delta_{nc}\neq\emptyset$ and $\Delta_{nc}=\emptyset$.

{\noindent\bf Case I.~} $\Delta_{nc}\neq\emptyset$.\stepcounter{theorem}

In this case, we prove that $q_f\leq 1$ and
\begin{equation}\label{boundofgwhens_0neq0}
g\leq\left\{
\begin{aligned}
 &3,\qquad &&\text{if~} q_f=0;\\[0.15cm]
&7,&&\text{if~} q_f=1.
\end{aligned}\right.
\end{equation}

First we prove $q_f\leq 1$.
Assume that $q_f\geq 2$. Then by Propositions \ref{fibredprop} and \ref{q_f>0fibred}, we get a morphism
$\tilde f':\, \wt S \to B'$ with $g(B')=q_f\geq 2$,
where $\wt S \to S$ is the blowing-up of $S$ centered at those points fixed by the hyperelliptic involution.
Clearly $\tilde f'$ factors through $\wt S\to S$,
hence we obtain a morphism $f':\, S \to B'$.
It is easy to see that
the restricted map $f'|_F:\, F \to B'$ is surjective, where $F$ is any fibre of $f:\,S \to B$.
Let $F_0$ be a singular fibre of $f$ over $\Delta_{nc}$.
Then $F_0$ has a non-compact Jacobian by assumption.
As $f$ reaches the Arakelov equality,
by \cite[Corollary 1.5]{ltyz13} and its proof, one gets that the geometric genus of $F_0$ is $g(F_0)=q_f$.
As $f'|_{F_0}$ is surjective, there is at least one irreducible component of $F_0$,
saying $C$, mapped surjectively onto $B'$. Hence $g(C)\geq g(B')=q_f$ by Hurwitz formula.
Thus
\begin{equation}\label{section6linshi1}
 g(F_0)=g(C)=g(B')=q_f,
\end{equation}
and $C$ is a section of $f':\,S \to B'$ since $q_f\geq 2$.
By \eqref{section6linshi1}, we see that any component of $F_0$ other
than $C$ is rational and hence contracted by $f'$.
This implies $\deg \left(f'|_F\right)=\deg \left(f'|_{F_0}\right)=1$
for any general fibre $F$ of $f$. Hence
$f'|_F$ is an isomorphism between $F$ and $B'$ for a general fibre $F$ of $f$.
It follows that $q_f=g(B')=g(F)=g$, which is a contradiction to \eqref{boundofq_f}.
Therefore $q_f \leq 1$.
\vspace{0.1cm}

Now we prove \eqref{boundofgwhens_0neq0}.
If $q_f=0$, then by \eqref{arakelovinequalityg-q_f}, the Higgs field associated to $f$
is strictly maximal.
So \eqref{boundofgwhens_0neq0} follows from Theorem \ref{mainthm1},
in which we prove that $g\leq 3$ for arbitrary families if $\Delta_{nc}\neq \emptyset$.
It remains to consider the case $q_f=1$.
According to \eqref{sigma_ncbukong<=},
one gets
$$0> \frac{g-8}{g+1}\delta_1(\Upsilon)+\frac{4g-21}{g+1}\delta_h(\Upsilon).$$
This implies that $g<8$, i.e., $g\leq 7$ as required.

\vspace{0.2cm}
{\noindent\bf Case II.~} $\Delta_{nc}=\emptyset$.

In this case, we prove that $q_f\leq 3$ and
\begin{equation}\label{boundofgwhens_0=0}
g\leq\left\{
\begin{aligned}
 &4,\qquad &&\text{if~} q_f=0;\\
&5,&&\text{if~} q_f=1\text{~or~} 2;\\
&6,&&\text{if~} q_f =3.
\end{aligned}\right.
\end{equation}

We divide the proof into three subcases:

\vspace{0.1cm}
{\noindent\sc Subcase A: $q_f\leq 1$.}

By \eqref{sigma_nckong<=}, we get
\begin{eqnarray*}
0&\geq&
\left(\frac{4(2g+1-3q_f)(g-1)}{(2g+1)(g-q_f)}-3\right) \delta_1(\Upsilon)+
\sum_{i=2}^{[g/2]}
\left(\frac{4(2g+1-3q_f)i(g-i)}{(2g+1)(g-q_f)}-4\right) \delta_i(\Upsilon).\\
&=&
\left\{\begin{aligned}
&\frac{g-4}{g}\cdot\big(\delta_1(\Upsilon)+4\delta_h(\Upsilon)\big),&\qquad& \text{if~}q_f=0;\\
&\frac{2g-11}{2g+1}\cdot \delta_1(\Upsilon)+\frac{8g-36}{2g+1}\cdot \delta_h(\Upsilon),&&\text{if~}q_f=1.
\end{aligned}\right.
\end{eqnarray*}
Note that
\begin{equation}\label{bothnot=0}
\text{$\delta_1(\Upsilon)\geq 0$, $\delta_h(\Upsilon)\geq 0$,
~and they cannot be zero simultaneously by \eqref{formulaofdegomega}.}
\end{equation}
Hence
$$g\leq 4,\quad \text{if~}q_f=0; \qquad\qquad g\leq 5,\quad \text{if~}q_f=1.$$

\vspace{0.1cm}
{\noindent\sc Subcase B: $q_f=2$.}

In this subcase, \eqref{sigma_nckong<=} reads as
\begin{equation}\label{sigma_nckong<=q_f=2}
0\geq \frac{2g^2-19g+26}{(2g+1)(g-2)} \cdot \delta_1(\Upsilon)+
\sum_{i=2}^{[g/2]}
\left(\frac{4(2g-5)i(g-i)}{(2g+1)(g-2)}-4\right) \delta_i(\Upsilon).
\end{equation}
When $g \geq 8$, it is easy to show that
$$\frac{4(2g-5)i(g-i)}{(2g+1)(g-2)}-4 > \frac{2g^2-19g+26}{(2g+1)(g-2)}>0,
\qquad \forall~2\leq i\leq [g/2].$$
But this is impossible by \eqref{bothnot=0} and \eqref{sigma_nckong<=q_f=2}.
Hence we may assume $g\leq 7$. So
$$\frac{2g^2-19g+26}{(2g+1)(g-2)}<0.$$
As $q_f=2$, by \eqref{eqnlowerSigma_0=0}, one has
\begin{equation}\label{Sigma_0=0q_f=2s_3}
\delta_1(\Upsilon) \leq \sum_{i=2}^{[g/2]} \frac{(2i+1)(2g+1-2i)}{12(g+1)} \delta_i(\Upsilon).
\end{equation}
Combining this with \eqref{sigma_nckong<=q_f=2}, we obtain
\begin{eqnarray*}
0&\geq& \sum_{i=2}^{[g/2]}\left(\frac{4(2g-5)i(g-i)}{(2g+1)(g-2)}-4
+\frac{2g^2-19g+26}{(2g+1)(g-2)}\cdot \frac{(2i+1)(2g+1-2i)}{12(g+1)}\right) \delta_i(\Upsilon)\\
&=&\sum_{i=2}^{[g/2]}\left(\frac{26g^2-55g-34}{\,3(g+1)(2g+1)(g-2)\,}\cdot i(g-i)-4
+\frac{2g^2-19g+26}{12(g+1)(g-2)}\right) \delta_i(\Upsilon)\\
&\geq &\left(\frac{26g^2-55g-34}{\,3(g+1)(2g+1)(g-2)\,}\cdot 2(g-2)-4
+\frac{2g^2-19g+26}{12(g+1)(g-2)}\right) \delta_h(\Upsilon)\\
&= &\left(\frac{28g^2-146g-80}{3(g+1)(2g+1)}
+\frac{2g^2-19g+26}{12(g+1)(g-2)}\right) \delta_h(\Upsilon).
\end{eqnarray*}
Thus if $g\geq 6$, then it follows that $0\geq \delta_h(\Upsilon)$, so $\delta_h(\Upsilon)=0$.
According to \eqref{Sigma_0=0q_f=2s_3}, we get $\delta_1(\Upsilon)=0$ too.
However, it is a contradiction by \eqref{bothnot=0}.
Therefore, we have proved that $g\leq 5$ if $q_f=2$.

\vspace{0.1cm}
{\noindent\sc Subcase C: $q_f\geq 3$.}

In this subcase, it suffices to prove $g\leq 6$,
from which it follows that $q_f=3$ by Proposition \ref{q_f>0fibred}.
To our purpose, we assume $g\geq 7$ in the rest.
We will deduce a contradiction, and so complete the proof.

Since $\Delta_{nc}=\emptyset$, $b=g(B)\geq 2$ by \eqref{arakelovinequalityg-q_f}.
Let $d$ be the degree of the Albanese map $S\to {\rm Alb\,}(S)$.
According to Proposition \ref{q_f>0fibred} and Remark \ref{boundofq_fremark}, it is known that $d\geq 2$.

If $d=2$, then by Lemma\,\ref{d=2bars_i=0},
\begin{equation}\label{d=2bars_i=0'}
\delta_{i}(\Upsilon)=0,\qquad\forall~1\leq i \leq q_f-1.
\end{equation}
Note that by a remarkable result of Xiao (cf. \cite{xiao92-0} or Remark \ref{boundofq_fremark}),
$q_f\leq \frac g2$ since $f$ is non-isotrivial. So
\begin{eqnarray*}
\frac{4(2g+1-3q_f)q_f}{2g+1}
&\geq& \min\left\{\frac{4(2g+1-3\cdot3)\cdot3}{2g+1},\,
\frac{4(2g+1-3\cdot\frac{g}{2})\cdot\frac{g}{2}}{2g+1}\right\}\\
&\geq&\frac{21}{5}, \qquad\qquad\text{since we assume that~}g\geq 7.
\end{eqnarray*}
Hence according to \eqref{sigma_nckong<=}, we get
\begin{eqnarray*}
0 &\geq& \sum_{i=q_f}^{[g/2]} \left(\frac{4(2g+1-3q_f)i(g-i)}{(2g+1)(g-q_f)}-4\right) \delta_i(\Upsilon)\\
&\geq& \sum_{i=q_f}^{[g/2]} \left(\frac{4(2g+1-3q_f)q_f(g-q_f)}{(2g+1)(g-q_f)}-4\right) \delta_i(\Upsilon)\\
&\geq& \frac15 \sum_{i=q_f}^{[g/2]}  \delta_i(\Upsilon).
\end{eqnarray*}
By \eqref{bothnot=0} and \eqref{d=2bars_i=0'},
we see that it is impossible.

Finally we assume $d\geq 3$. By \eqref{boundofq_f} and Remark \ref{boundofq_fremark},
$q_f \leq \frac{g+1}{3}$.
It follows that $g\geq 3q_f-1\geq 8$.
According to \eqref{sigma_nckong<=}, we get
\begin{equation}\label{q_fgeq3dgeq3linshi1}
\begin{aligned}
&\sum_{i=2}^{[g/2]} 4\Big((2g+1-3q_f)i(g-i)-(2g+1)(g-q_f)\Big) \cdot \delta_i(\Upsilon)\\[0.1cm]
\leq~& \big(-2g^2+(7+6q_f)g+4-15q_f\big)\cdot \delta_1(\Upsilon).
\end{aligned}
\end{equation}
Note that for $2\leq i\leq [g/2]$,
\begin{eqnarray*}
(2g+1-3q_f)i(g-i)-(2g+1)(g-q_f)
&\geq& (2g+1-3q_f)\cdot2\cdot(g-2)-(2g+1)(g-q_f)\quad\\
&\geq& (2g+1)(g-4)-(4g-13)q_f\\
&\geq& (2g+1)(g-4)-(4g-13)\cdot\frac{g+1}{3}\\
&=&\frac{2g^2-12g+1}{3}>0,\qquad\text{since~}g\geq 8
\end{eqnarray*}
So by \eqref{q_fgeq3dgeq3linshi1}, we have in particular
$$\Theta\triangleq-2g^2+(7+6q_f)g+4-15q_f\geq 0.$$
Combining \eqref{q_fgeq3dgeq3linshi1} with \eqref{eqnlowerSigma_0=0}, we obtain
\begin{eqnarray*}
&&\sum_{i=2}^{[g/2]} 4\Big((2g+1-3q_f)i(g-i)-(2g+1)(g-q_f)\Big) \cdot \delta_i(\Upsilon)\\
&\leq& \Theta\cdot
\left(\frac{1}{12}\cdot\sum_{i=q_f}^{[g/2]} \frac{(2i+1)(2g+1-2i)}{g+1} \delta_i(\Upsilon)
-\frac{1}{12}\cdot\sum_{i=2}^{q_f-1} 4i(2i+1) \delta_i(\Upsilon)\right)
\end{eqnarray*}
i.e.,
\begin{equation}\label{q_fgeq3dgeq3linshi2}
0\geq \sum_{i=2}^{q_f-1} a_i \delta_i(\Upsilon) + \sum_{i=q_f}^{[g/2]} b_i \delta_i(\Upsilon),
\end{equation}
where
\begin{eqnarray*}
a_i&=&4\big((2g+1-3q_f)i(g-i)-(2g+1)(g-q_f)\big)
+\frac{\Theta}{3}\cdot i(2i+1),\\[0.1cm]
b_i&=&4\big((2g+1-3q_f)i(g-i)-(2g+1)(g-q_f)\big)
-\frac{\Theta}{12}\cdot \frac{(2i+1)(2g+1-2i)}{g+1}.
\end{eqnarray*}
It is not difficult to check that $a_i >0$ for $1\leq i \leq q_f-1$.
For $b_i$ with $q_f\leq i \leq [g/2]$, we have
$$
\begin{aligned}
b_i &=4\left((2g+1-3q_f)-\frac{\Theta}{12(g+1)}\right)
\cdot i(g-i)-4(2g+1)(g-q_f)-\frac{\Theta\cdot(2g+1)}{12(g+1)}\\
&\geq 4\left((2g+1-3q_f)-\frac{\Theta}{12(g+1)}\right)
\cdot q_f\cdot (g-q_f)-4(2g+1)(g-q_f)-\frac{\Theta\cdot(2g+1)}{12(g+1)} \\
&= \frac{(2g+1)(g-q_f)}{12(g+1)}\Bigg(4q_f(13g-21q_f+8)-50g-51\\
&\hspace{4cm}+\frac{4\big((q_f-1)g+(g+1-3q_f)(g-1)\big)}{g-q_f}\Bigg)\\
&\geq \frac{(2g+1)(g-q_f)}{12(g+1)}\cdot\Big( 4q_f(13g-21q_f+8)-50g-51\Big).
\end{aligned}
$$
Let $f(g,q_f)=4q_f(13g-21q_f+8)-50g-51$. Since $g\geq 8$, one gets
\begin{eqnarray*}
f(g,3)&=&106g-711>0,\\
f\left(g,\frac{g+1}{3}\right)&=& \frac13\cdot(24g^2-122g-149)>0.
\end{eqnarray*}
Hence for any $3\leq q_f \leq \frac{g+1}{3}$, we have
$$f(g,q_f)\geq \min\left\{f(g,3),~f\left(g,\frac{g+1}{3}\right)\right\}>0.$$
Thus for any $3\leq q_f \leq \frac{g+1}{3}$,
$$b_i >0.$$
Now from \eqref{q_fgeq3dgeq3linshi2} it follows that $\delta_i(\Upsilon)=0$ for all $2\leq i \leq [g/2]$.
By \eqref{eqnlowerSigma_0=0}, $\delta_1(\Upsilon)=0$ too.
This is a contradiction by \eqref{bothnot=0}.
Therefore the proof is complete.
\end{proof}

\section{Examples}\label{sectionexample}
In this section, we construct two Shimura curves
contained generically in the Torelli locus of hyperelliptic
curves of genus $3$ and $4$ respectively.

The idea is to construct first non-isotrivial semi-stable families of hyperelliptic curves
of genus $g=3$ and $4$ respectively
by taking double covers of ruled surfaces branched over suitable branched locus.
Then we show that their corresponding Jacobian families reach the Arakelov bound.
By \cite{vz04} (or \cite{mvz12}), a semi-stable one-dimensional family of $g$-dimensional abelian variety
reaching the Arakelov bound gives a Kuga curve in $\cala_g$.
Hence we obtain two Kuga curves contained generically in the Torelli locus of hyperelliptic
curves of genus $g=3$ and $4$ respectively.

To show that those two Kuga curves are indeed Shimura curves,
first we note that the Higgs field associated to the family is actually strictly maximal for $g=4$,
hence by virtue of \cite{vz04}, it is a Shimura curve.
For $g=3$, we present two ways to prove that such a Kuga curve is also Shimura.

\begin{example}\label{exshimurag=3}
Shimura curve contained generically in the Torelli locus of hyperelliptic curves of genus $g=3$.
\vspace{0.2cm}

Let $C_0,\,H_{x_0}\subseteq X_0=\bbp^1\times \bbp^1$ be defined respectively by
$$1+(4t-2)x^2+x^4=0,\text{\quad and \quad} x=x_0,$$
where $t$ and $x$ are the coordinates of the first and second factor of $X_0$ respectively.
The projection of $C_0$ to the first factor $\bbp^1$ of $X_0$
branches exactly over three points, i.e., $\{0,1,\infty\}$.
Locally, it looks like the following.
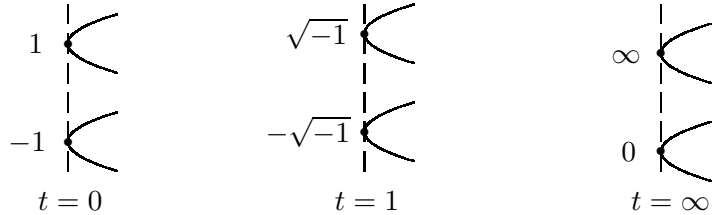
\begin{center}
\setlength{\unitlength}{1.3mm}
\begin{picture}(90,20)
\multiput(15,2)(0,3){6}{\line(0,1){2}}
\qbezier(20,2)(10,5)(20,8)
\qbezier(20,12)(10,15)(20,18)
\put(15,5){\circle*{0.8}}
\put(15,15){\circle*{0.8}}
\put(12,-2){$t=0$}
\put(9,4){$-1$}
\put(11,14){$1$}

\multiput(45,2)(0,3){6}{\line(0,1){2}}
\qbezier(50,3)(40,6)(50,9)
\qbezier(50,13)(40,16)(50,19)
\put(45,6){\circle*{0.8}}
\put(45,16){\circle*{0.8}}
\put(42,-2){$t=1$}
\put(35,5){$-\sqrt{-1}$}
\put(37,15){$\sqrt{-1}$}

\multiput(75,2)(0,3){6}{\line(0,1){2}}
\qbezier(80,1)(70,4)(80,7)
\qbezier(80,11)(70,14)(80,17)
\put(75,4){\circle*{0.8}}
\put(75,14){\circle*{0.8}}
\put(72,-2){$t=\infty$}
\put(71,3){$0$}
\put(70,13){$\infty$}
\end{picture}
\end{center}
\vspace{0.3cm}

Let $\varphi:\,\bbp^1 \to \bbp^1$ be the cyclic cover of degree $4$
defined by $t=(t')^4$,
totally ramified over $\{0,\infty\}$.
Let $X_1$ be the normalization of the fibre-product $X_0\times_{\bbp^1}\bbp^1$
and $R$ the inverse image of
$$C_0\cup H_{1}\cup H_{-1}\cup H_0\cup H_{\infty}\,.$$
Then $R$ is a double divisor,
i.e.,  we can construct a double cover $S_1 \to X_1$ branched exactly over $R$.
Let $S' \to X_r$ be the canonical resolution,
and $f:\,S \to \bbp^1$ the relatively minimal smooth model as follows.
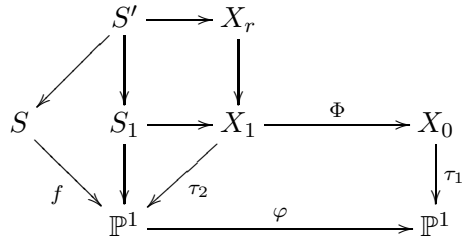
\begin{center}\mbox{}
\xymatrix{
&S'\ar[d]\ar[ld]\ar[r] &X_r \ar[d] &&\\
S  \ar[dr]_{f} & S_1\ar[r]\ar[d]
& X_1 \ar[rr]^{\Phi}  \ar[dl]^{\tau_2} && X_0 \ar[d]^{\tau_1}  \\
& \bbp^1 \ar[rrr]^{\varphi} &&&\bbp^1
}
\end{center}

By the theory of double covers (cf. \cite[\S\,III.22]{bhpv}),
it is not difficult to show that $f:\,S \to \bbp^1$
is a semi-stable hyperelliptic family of genus $g=3$.
In fact, there are exactly $6$ singular fibres in the family $f$,
i.e., those fibres $\Upsilon$ over $\Delta:=\varphi^{-1}(0\cup1\cup\infty)$.
More precisely, for any fibre $F$ over $\varphi^{-1}(1)$,
$F$ is an irreducible singular elliptic curve with exactly two nodes, hence
$$\xi_0(F)=2,\quad \text{~and~} \quad \delta_1(F)=\xi_1(F)=0;$$
for any fibre $F$ over $\varphi^{-1}(0\cup\infty)$,
$F$ is a chain of three smooth elliptic curves, hence
$$\delta_1(F)=2,\quad \text{~and~}\quad \xi_0(F)=\xi_1(F)=0.$$
So
$$\xi_0(\Upsilon)=8,\qquad \delta_1(\Upsilon)=4,\quad\text{and}\quad \xi_1(\Upsilon)=0.$$
Therefore by \eqref{formulanoether}, \eqref{formulaofdegomega} and \eqref{formulaofdelta_f},
\begin{equation*}
\delta_f=12,\qquad
\deg \big(f_*\Omega^1_{S/\bbp^1}(\log\Upsilon)\big)=2,\qquad
\omega_{S/\bbp^1}^2=12.
\end{equation*}

By definition, those fibres over $\Delta_c:=\varphi^{-1}(0\cup\infty)$ have compact Jacobians,
while those over $\Delta_{nc}:=\varphi^{-1}(1)$ have non-compact Jacobians.
Hence the Jacobian of $f$ admits exactly $\#\left(\Delta_{nc}\right)=4$ singular fibres over $\bbp^1$.
By \cite[\S\,7]{vz04}, the Higgs field of any semi-stable family of abelian varieties over $\bbp^1$
with exactly $4$ singular fibres must be maximal.
i.e., it reaches the Arakelov bound.
Hence $f$ is a Kuga family.

Let $\mathbb F^{1,0}\oplus\mathbb F^{0,1}$ be the flat part of
the logarithmic Higgs bundle associated to the VHS of the Jacobian of $f$ as on Page \pageref{defoftheta}.
Since the base of the family is $\bbp^1$,
$q_f=\rank \mathbb F^{1,0}$.
Hence
$$\deg \big(f_*\Omega^1_{S/\bbp^1}(\log\Upsilon)\big)=
\frac{g-q_f}{2}\cdot\deg \left(\Omega^1_{\bbp^1}(\log\Delta_{nc})\right)=3-q_f,$$
from which it follows that $q_f=1$.

It remains to show that $f$ is in fact a Shimura family.
We present here two ways.

(i).~
Since the base $\bbp^1$ of $f$ is simply connected, by \cite[Theorem 0.2]{vz04},
the Jacobian of $f$ is isogenous over $\bbp^1$ to a product
\begin{equation}\label{exampledec}
E\times_{\bbp^1}\mathcal E\times_{\bbp^1}\mathcal E,
\end{equation}
where $E$ is a constant elliptic curve, and $\mathcal E \to \bbp^1$ is
a family of semi-stable elliptic curves reaching the Arakelov bound.
To show that $f$ is a Shimura family, it suffices to prove that
the constant part $E$ has a complex multiplication.

It is not difficult to see that our family is actually defined by
\begin{equation}\label{exampledefoff}
y^2=\big(1+(4(t')^4-2)x^2+x^4\big)\cdot(x^2-1)\cdot x.
\end{equation}
Let $E_0$ be a constant elliptic curve defined by
$$u^4=v\cdot(v+1)^2.$$
Then it is clear that $E_0$ has complex multiplication by $\mathbb Z\left[\sqrt{-1}\right]$.
Define a morphism from the family $f$ to the constant family $E_0$ by
$$(u,\,v)=\psi(x,\,y)=\left(\frac{\sqrt{2}\cdot t'y}{(x^2-1)^2},\,\frac{4(t')^4x^2}{(x^2-1)^2}\right).$$
It can be checked easily that $\psi$ is well-defined.
Hence the Jacobian of $f$ contains a constant part $E_0$.
Note that the constant part $E$ in the decomposition \eqref{exampledec} is unique up to isogenous,
and the property with a complex multiplication is invariant under isogenous.
Therefore, the constant part $E\sim E_0$ has a complex multiplication, and so $f$ is a Shimura family.

(ii).~
We prove that $f$ is a Shimura family by showing that our family $f$ is actually isomorphic to
a known Shimura family constructed by Moonen and Oort \cite{mo11}.

Let
$$u=\frac{1+x^2}{1-x^2},\qquad v=\frac{2y}{(1-x^2)^2},\qquad  w=\left(\frac{1+x^2}{1-x^2}\right)^2.$$
Then by virtue of \eqref{exampledefoff}, we see that our family is isomorphic to
$$ {\mathcal U_{t'}}:\quad \left\{
\begin{aligned}
u^2&=w,\\
v^4&= \Big(2(t')^4w-2\big((t')^4-1\big)\Big)^2\cdot(w-1).
\end{aligned}\right.$$
Such a family can be viewed as a family of abelian covers of $\bbp^1$ branched exactly over
$4$ points with Galois group $\mathbb Z_2\times \mathbb Z_4$ and local monodromy of
the branched points being $\big((1,0), (1,1), (0,1), (0,2)\big)$.
And it is just the family (22) given in \cite[\S 6,\,{\sc Table\,2}]{mo11}, which is Shimura.
So is $f$.

We remark that by \cite{mo11}, we do not know whether the corresponding Shimura curve
is complete or not (i.e., whether $\Delta_{nc}=\emptyset$ or not).
Our concrete description shows that such a Shimura curve is a non-complete rational Shimura curve.
\end{example}

\begin{example}\label{exshimurag=4}
Shimura curve with strictly maximal Higgs field
contained generically in the Torelli locus of hyperelliptic curves of genus $g=4$.
\vspace{0.15cm}

The construction is similar to Example \ref{exshimurag=3}.

Let $C_0$, $H_{x_0}$ and $X_0$ be the same as those in Example \ref{exshimurag=3}.
Let $\pi:\,B \to \bbp^1$ be a cover of degree $8$,
ramified uniformly over $\{0,1,\infty\}$ with ramification indices equal to $4$.
It is easy to see that such a cover exists, and
$$g(B)=2,\qquad \#(\Delta)=6,$$
where $\Delta=\varphi^{-1}(0\cup1\cup\infty)$.
Let $X_1$ be the normalization of $X_0\times_{\bbp^1}B$
and $R$ the inverse image of
$$C_0\cup H_{1}\cup H_{-1}\cup H_{\sqrt{-1}}\cup H_{-\sqrt{-1}}\cup H_0\cup H_{\infty}\,.$$
Then $R$ is a double divisor,
i.e.,  we can construct a double cover $S_1 \to X_1$ branched exactly over $R$.
Let $f:\,S \to B$ the relatively minimal smooth model as follows.
\begin{center}\mbox{}
\xymatrix{
S  \ar[dr]_{f}\ar@{<-->}[r] & S_1\ar[r]\ar[d]
& X_1 \ar[rr]^{\Phi}  \ar[dl]^{\tau_2} && X_0 \ar[d]^{\tau_1}  \\
& B \ar[rrr]^{\varphi} &&&\bbp^1
}
\end{center}

By \cite[\S\,III.22]{bhpv}, one can show that $f:\,S \to B$
is a semi-stable hyperelliptic family of genus $g=4$
with $6$ singular fibres,
i.e., those fibres $\Upsilon$ over $\Delta$.
More precisely, for any fibre $F\in \Upsilon$,
$F$ consists of two smooth elliptic curves $D_1$, $D_2$, and a smooth curve $\wt D$ of genus\,$2$,
such that $D_1$ does not intersect $D_2$, and $\wt D$ intersects each $D_i$ in one point for $i=1,2$.
Hence
$$\delta_1(F)=2,\quad \text{~and~} \quad \delta_2(F)=\xi_0(F)=\xi_{1}(F)=0.$$
So
$$\delta_1(\Upsilon)=12,\quad \text{~and~} \quad \delta_2(\Upsilon)=\xi_0(\Upsilon)=\xi_{1}(\Upsilon)=0.$$
Therefore by \eqref{formulanoether}, \eqref{formulaofdegomega} and \eqref{formulaofdelta_f},
\begin{equation*}
\delta_f=12,\qquad
\deg \big(f_*\Omega^1_{S/B}(\log\Upsilon)\big)=4,\qquad
\omega_{S/B}^2=36.
\end{equation*}

By definition, any singular fibre of $f$ has a compact Jacobian,
so the Jacobian of $f$ is a smooth family of abelian varieties of dimension $4$.
Let $A^{1,0}\subseteq f_*\Omega^1_{S/B}(\log\Upsilon)$
be the ample part described on Page \pageref{defoftheta}.
Then according to Arakelov inequality, we have
$$4=\deg \big(f_*\Omega^1_{S/B}(\log\Upsilon)\big)
\leq \frac{\rank A^{1,0}}{2}\cdot \deg \Omega_B=\rank A^{1,0}
\leq \rank f_*\Omega^1_{S/B}(\log\Upsilon) =g=4.$$
Hence the Jacobian of $f$ reaches the Arakelov bound with $\rank A^{1,0}=g$,
i.e., the Higgs field associated to $f$ is strictly maximal.
Therefore $f$ is a Shimura family.

We remark that in this example,
$$c_1^2\left(\Omega_S^1(\log D)\right)=3c_2\left(\Omega_S^1(\log D)\right)=72,$$
where $D$ is the union of those $12$ smooth disjoint elliptic curves contained in $\Upsilon$.
Hence $S\setminus D$ is a ball quotient by \cite{kobayashi} or \cite{mok12}.
\end{example}

\vspace{0.3cm}
\noindent{\bf Acknowledgements.}
\phantomsection
\addcontentsline{toc}{section}{Acknowledgements}
We would like to thank Chris Peters,  Guitang Lan and Jinxing Xu
for discussions on the topic related to global invariant
cycles with locally constant coefficients.
Especially, the proof of Lemma \ref{lemmapeters} comes from a discussion with Chris Peters.
We would also like to thank Shengli Tan and Hao Sun for discussing with us on
Miyaoka-Yau's inequality and the slope inequality.
We are grateful to Yanhong Yang for her interests,
careful reading and valuable suggestions.

\phantomsection

\vspace{.3cm}
\phantomsection
\addcontentsline{toc}{section}{Contact}
\noindent
{\it E-mail address}: lvxinwillv@gmail.com\\[.1cm]
{\it E-mail address}: zuok@uni-mainz.de

\vspace{0.5cm}
\noindent {\sc Institut f\"ur Mathematik, Universit\"at Mainz, Mainz, 55099, Germany}

\begin{thebibliography}{99}
\addcontentsline{toc}{section}{References}

\bibitem{bar98} M. \'A. Barja,
{\it On the slope and geography of fibred surfaces and threefolds},
Ph.D. Thesis of Universitat de Barcelona. (1998).

\bibitem{bhpv} W. P. Barth, K. Hulek, C. A. M. Peters and A. Van de Ven,
{\itshape Compact Complex Surfaces},
(Second Enlarged Edition), A Series of Modern Surveys in Mathematics,
Volume {\bf 4}, Springer-Verlag, (2004).

\bibitem{ch88}
M. Cornalba and J. Harris,
{\it Divisor classes associated to families of stable varieties,
with applications to the moduli space of curves},
Ann. Sc. Ec. Norm. Sup., {\bf 21} (1988), 455--475.

\bibitem{deligne71} P. Deligne,
{\it Th\'{e}orie de Hodge II},
Publ. Math. I.H.E.S., {\bf40} (1971), 5--58.

\bibitem{dm69} P. Deligne and D. Mumford,
{\it The irreducibility of the space of curves of given genus},
Inst. Hautes \'etudes Sci. Publ. Math., {\bf 36} (1969), 75--109.

\bibitem{do86} B. Dwork and A. Ogus,
{\it Canonical liftings of Jacobians},
Compositio Math. {\bf58} (1986), 111--131.

\bibitem{faltings83}
G. Faltings,
{\it Arakelov's theorem for abelian varieties},
Invent. math. {\bf73} (1983), 337--348.

\bibitem{fujita78} T. Fujita,
{\it On K\"ahler fiber spaces over curves},
J. Math. Soc. Japan {\bf30(4)} (1978), 779--794.

\bibitem{hain99} R. Hain,
{\it Locally symmetric families of curves and Jacobians},
Moduli of curves and abelian varieties, 91--108, Aspects Math., E33, Vieweg, (1999).

\bibitem{harrismumford82} J. Harris and D. Mumford,
{\it On the Kodaira dimension of the moduli space of curves},
With an appendix by William Fulton.
Invent. Math. {\bf 67(1)} (1982), 23--88.

\bibitem{hartshorne} R. Hartshorne,
{\it Algebraic Geometry},
GTM 52, Springer-Verlag, (1997).

\bibitem{djongnoot89}
J. de Jong and R. Noot,
{\it Jacobians with complex multiplication},
Arithmetic algebraic geometry (Texel, 1989), 177--192, Progr. Math., 89,
Birkh\"auser Boston, Boston, MA, (1991).

\bibitem{djongzh07} J. de Jong and S.-W. Zhang,
{\it Generic Abelian Varieties with Real Multiplication are not Jacobians},
Diophantine geometry, CRM Series, {\bf 4} (2007), 165--172.

\bibitem{jz02} J. Jost and K. Zuo,
{\it Arakelov type inequalities for Hodge bundles over algebraic varieties
I. Hodge bundles over algebraic curves},
J. Algebraic Geom. {\bf 11(3)} (2002), 535--546.

\bibitem{khashin83} S. I. Khashin,
{\it The irregularity of double surfaces},
Mathematical notes of the Academy of Sciences of the USSR,
{\bf 33(3)} (1983), 233--235.

\bibitem{kobayashi} R. Kobayashi,
{\it Einstein-Kaehler metrics on open algebraic surfaces of general type},
Tohoku Math. J.  {\bf 37(2)} (1985), no. 1, 43--77.


\bibitem{kollar87} J. Koll\'ar,
{\it Subadditivity of the Kodaira dimension: fibers of general type},
Algebraic Geometry, Sendai, 1985. Advanced Studies in Pure Mathematics {\bf 10} (1987), 361--398.

\bibitem{kukulies10} S. Kukulies,
{\it On Shimura curves in the Schottky locus},
J. Algebraic Geom. {\bf 19(2)} (2010), 371--397.

\bibitem{liukefeng96} K.-F. Liu,
{\it Geometric height inequalities},
Math. Res. Lett. {\bf 3(5)} (1996), 693--702.


\bibitem{ltyz13} J. Lu, S.-L. Tan, F. Yu and K. Zuo,
{\it A new inequality on the Hodge number $h^{1,1}$ of algebraic surfaces},
arXiv:1303.2749, (2013), to appear in Math. Z.

\bibitem{miyaoka84} Y. Miyaoka,
{\itshape The maximal number of quotient singularities
on surfaces with given numerical invariants},
Math. Ann. {\bf 268} (1984), 159--171.

\bibitem{mok12} N. Mok,
{\it Projective algebraicity of minimal compactifications of complex-hyperbolic
space forms of finite volume},
Perspectives in analysis, geometry, and topology, 331--354,
Progr. Math., 296, Birkh\"auser/Springer, New York, 2012.

\bibitem{moller11} M. M\"oller,
{\it Shimura and Teichm\"uller curves},
J. Mod. Dyn. {\bf 5(1)} (2011), 1--32.

\bibitem{mvz12} M. M\"oller, E. Viehweg and K. Zuo,
{\it Stability of Hodge bundles and a numerical characterization of Shimura varieties},
J. Differential Geom. {\bf 92(1)} (2012), 71--151.

\bibitem{moonen10} B. Moonen,
{\it Special subvarieties arising from families of cyclic covers of the projective line},
Doc. Math. {\bf 15} (2010), 793--819.

\bibitem{mo11} B. Moonen and F. Oort,
{\it The Torelli locus and special subvarieties},
arXiv:1112.0933.
Submitted to the Handbook of Moduli (G. Farkas and
I. Morrisson, eds).

\bibitem{moriwaki98} A. Moriwaki,
{\it Relative Bogomolov's inequality and the cone of
positive divisors on the moduli space of stable curves},
J. Amer. Math. Soc. {\bf 11(3)} (1998), 569--600.

\bibitem{mumford69} D. Mumford,
{\it A Note of Shimura's Paper ``Discontinuous Groups and Abelian Varieties"},
Math. Ann. {\bf 181} (1969), 345--351.

\bibitem{mumford74} ------,
{\it Abelian varieties. With appendices by C. P. Ramanujam and Yuri Manin.}
Corrected reprint of the second (1974) edition.
Tata Institute of Fundamental Research Studies in Mathematics, 5.
Published for the Tata Institute of Fundamental Research, Bombay; by Hindustan Book Agency, New Delhi, 2008.

\bibitem{oort97} F. Oort,
{\it Canonical liftings and dense sets of CM-points},
in: Arithmetic geometry (Cortona, 1994), 228--234,
Sympos. Math., XXXVII, Cambridge Univ. Press, Cambridge, (1997).

\bibitem{os79} F. Oort and J. Steenbrink,
{\it The local Torelli problem for algebraic curves},
Journ\'ees de G\'eometrie Alg\'ebrique d'Angers, (1979), 157--204.

\bibitem{peterssaito12} C. A. M. Peters and M. Saito,
{\it Lowest weights in cohomology of variations of Hodge structure},
Nagoya Math. J. {\bf 206} (2012), 1--24.

\bibitem{peterssteenbrink} C. A. M. Peters and J. H. M. Steenbrink,
{\it Mixed Hodge structures},
Ergebnisse der Mathematik und ihrer Grenzgebiete. 3. Folge. A Series of Modern Surveys in Mathematics,
52. Springer-Verlag, Berlin, 2008. xiv+470 pp.

\bibitem{sakai80} F. Sakai,
{\it Semistable curves on algebraic surfaces and logarithmic pluricanonical maps},
Math. Ann. {\bf 254(2)} (1980), 89--120.

\bibitem{tan95} S.-L. Tan,
{\itshape The minimal number of singular fibers of a semistable curve over $\bbp^1$},
J. Algebraic Geom. {\bf 4(3)} (1995), 591--596.

\bibitem{vz03} E. Viehweg and K. Zuo,
{\itshape Families over curves with a strictly maximal Higgs field},
Asian J. Math. {\bf 7(4)} (2003), 575--598.

\bibitem{vz04} ------,
{\itshape A characterization of certain Shimura curves in the moduli
stack of abelian varieties},
J. Differential Geom. {\bf 66(2)} (2004), 233--287.

\bibitem{vojta88} P. Vojta,
{\itshape Diophantine inequalities and Arakelov theory},
in Lang, Introduction to Arakelov Theory, Springer-Verlag, (1988), 155--178.

\bibitem{xiao92-0} G. Xiao,
{\it Irregular families of hyperelliptic curves},
Algebraic geometry and algebraic number theory (Tianjin, 1989--1990), 152--156,
Nankai Ser. Pure Appl. Math. Theoret. Phys., 3, World Sci. Publ., River Edge, NJ, 1992.
\end{thebibliography}
\end{document}